\documentclass[12pt,reqno]{amsart}

\usepackage[usenames,dvipsnames]{xcolor}
\definecolor{brightcerulean}{rgb}{0.11, 0.67, 0.84}
\usepackage{amssymb}
\usepackage{fullpage}
\usepackage{url}
\usepackage{amsthm}
\usepackage{bm}
\usepackage{latexsym}
\usepackage{float}
\restylefloat{table}
\usepackage{amsmath}
\usepackage{multirow}
\usepackage{eufrak}
\usepackage{mathrsfs}
\usepackage[font={small,it}]{caption}
\usepackage{amscd}
\usepackage[all,cmtip]{xy}
\usepackage{graphicx}
\usepackage{amscd}
\usepackage{float}
\usepackage{graphics}
\usepackage{tikz}
\usepackage{tikz-cd}
\usepackage{comment}
\usepackage{xspace}
\usepackage{mathtools}

\usepackage{amssymb}
\usepackage{amsthm}
\usepackage{graphicx}
\usepackage{subfig}
\usepackage{enumerate}

\usepackage[pagebackref=true, colorlinks=false, allcolors=brightcerulean, allbordercolors=brightcerulean,
pdfauthor={J. S. Balakrishnan, J. Caro},
pdftitle={A refined Chabauty--Coleman bound for surfaces}
]{hyperref}

\usetikzlibrary{patterns,decorations.pathreplacing}
\usepackage{marginnote}

\theoremstyle{plain}

\newtheorem{theorem}{Theorem}[section]
\newtheorem{proposition}[theorem]{Proposition}
\newtheorem{lemma}[theorem]{Lemma}

\DeclareFontFamily{U}{wncy}{}
    \DeclareFontShape{U}{wncy}{m}{n}{<->wncyr10}{}
    \DeclareSymbolFont{mcy}{U}{wncy}{m}{n}
    \DeclareMathSymbol{\Sh}{\mathord}{mcy}{"58}
\theoremstyle{definition}

\newcommand{\appsection}[1]{\let\oldthesection\thesection
\renewcommand{\thesection}{Appendix \oldthesection}
\section{#1}\let\thesection\oldthesection}
\newtheorem{definition}[theorem]{Definition}
\newtheorem{algorithm}[theorem]{Algorithm}

\newtheorem{observation}[theorem]{Observation}
\theoremstyle{remark}

\newtheorem{remark}[theorem]{Remark}
\newtheorem{example}[theorem]{Example}

\DeclareMathOperator{\red}{red}

\DeclareMathOperator{\ch}{char}

\DeclareMathOperator{\Div}{Div}

\def\Z{{\mathbb{Z}}}
\def\F{{\mathbb{F}}}
\def\Q{{\mathbb{Q}}}

\def\P{{\mathbb{P}}}

\def\O{{\mathcal{O}}}

\def\mfrak{{\mathfrak{m}}}

\DeclareMathOperator{\ord}{ord}

\usepackage{microtype}

\DeclareMathOperator{\divisor}{div}
\DeclareMathOperator{\Ann}{Ann}

\DeclareMathOperator{\Span}{span}

\DeclareMathOperator{\Sing}{Sing}

\DeclareMathOperator{\Spec}{Spec}

\renewcommand*{\backref}[1]{}
\renewcommand*{\backrefalt}[4]{%
  \ifcase #1 %
    \relax
  \or
    $\uparrow$#2.%
  \else
    $\uparrow$#2.%
  \fi%
}

\begin{document}
\title{A refined Chabauty--Coleman bound for surfaces}

\author{Jennifer S. Balakrishnan}
\address{ Department of Mathematics \& Statistics, Boston University, 665 Commonwealth Avenue, Boston, MA 02215, USA}
\email[J. S. Balakrishnan]{jbala@bu.edu}%

\author{Jerson Caro}
\address{Department of Mathematics \& Statistics, Boston University, 665 Commonwealth Avenue, Boston, MA 02215, USA}
\email[J. Caro]{jlcaro@bu.edu}

\date{\today}

\maketitle

\begin{abstract}
Caro and Pasten gave an explicit upper bound on the number of rational points on a hyperbolic surface that is embedded in an abelian variety of rank at most one. We show how to use their method to produce a refined bound on the number of rational points on the surface $W_2\coloneqq C+C$ in the case of a hyperelliptic curve $C$ of genus $3$ over $\Q$. Combining this with work of Siksek, we use this to determine $W_2(\Q)$ in a selection of examples. 
\end{abstract}



\section{Introduction}
Let $C$ be a smooth projective curve of genus $g$ defined over $\Q$. If the Mordell-Weil rank $r$  of the Jacobian of $C$ is less than $g$, Coleman's work on effective Chabauty \cite{coleman} provides an upper bound on the number of rational points on $C$: 
$$
\#C(\Q) \leq \#C(\F_p) + 2g-2,
$$ 
where $p > 2g$ is a prime of good reduction. This bound arises from constructing an {annihilating differential} whose $p$-adic (Coleman) integral cuts out a finite set $C(\Q_p)_1$ of $p$-adic points on the curve containing the set of rational points. Bounding the size of $C(\Q_p)_1$ produces this upper bound.

The first extension of Coleman's explicit bound beyond the context of curves is due to Caro and Pasten \cite{CaroPasten2021}.  
They exhibit a bound for the number of rational points of a hyperbolic surface $ X $ defined over a number field, which is embedded in an abelian variety $ A $ of rank at most one, assuming certain conditions on the reduction type at a well-chosen prime. 

Here, we apply the Caro--Pasten approach to specific surfaces, refining the bound in these cases (cf. Theorem \ref{ThmMain}). We further give examples where the method can yield sharp results, in conjunction with some work of Siksek, while existing methods are not applicable.

Recall that when the curve $C$ admits a degree $d$ morphism to $\P^1$ defined over $\Q$, the $d$-th symmetric power $C^{(d)}$ of $C$ contains a $\P^1$, giving it infinitely many rational points. An \textit{unexpected degree $d$ point} is a rational point of $C^{(d)}$ that lies outside the $\P^1$s in $C^{(d)}$. For hyperelliptic curves of genus $3$, the unexpected quadratic points correspond to rational points of the surface $W_2\coloneqq C+C\subset J$, where $J$ denotes the Jacobian of $C$. In this manuscript, we present an algorithm, which follows the work of Caro and Pasten, to obtain an upper bound for $\#W_2(\Q)$. 

\begin{theorem}\label{ThmMain} Let $C$ be a hyperelliptic curve of genus $3$ given by an odd degree model defined over $\Q$  whose Jacobian has rank at most $1$.  Let  $p\ge 11$ be a prime of good reduction for $C$. Suppose that the reduction of $W_2$ does not contain elliptic curves over $\F_p^{\rm alg}$. Then  $W_2(\Q)$  is finite and
\begin{equation}\label{mainthm}
\# W_2(\Q)\le \# W_2(\F_p) +  2p+12\sqrt{p}+7. 
\end{equation}
\end{theorem}

In \cite{Siksek}, Siksek presented an explicit and practical procedure for computing algebraic points on curves. Siksek established a criterion to verify if a subset of the rational points of $C^{(d)}$ is the entire set, subject to the condition that the rank of the Jacobian is at most $g-d$. 

In this manuscript, we provide three examples where the Caro--Pasten method combined with ideas from Siksek's method enables us to compute $W_2(\Q)$ in cases where Siksek's method is not directly applicable (cf. Section \ref{Comparison}). These examples are achieved by explicitly computing the differentials used to produce the bound in \eqref{mainthm} and bounding the number of their zeros.

In Section 2, we recall a special case of the Caro--Pasten method for bounding rational points on surfaces and discuss how it gives a bound on the number of unexpected quadratic points on genus 3 curves of Jacobian rank 1.  We also review Siksek's work on computing algebraic points on curves. 
In Section 3, we refine the Caro--Pasten method for genus 3 {hyperelliptic} curves. In particular, we improve the applicability of the method from previously requiring $p \geq 521$ to $p \geq 11$ (or $p \ge 5$ in some cases). We also give improved upper bounds on the number of points in each residue disk of the curve contained in the intersection of $W_2(\Q_p)$ with the $p$-adic closure of $J(\Q)$ in $J(\Q_p)$.
In Section 4, we distill the above into an algorithm, and in Section 5, we present three examples where applying the algorithm and combining it with Siksek's method allows us to compute the set $W_2(\Q)$.

\begin{remark} We restrict to genus 3 hyperelliptic curves $C$ of Jacobian rank 1 for the following reasons: computing the rank of the Mordell--Weil group of the Jacobian of a generic genus $3$ curve is difficult \cite{bruin-poonen-stoll}, while for hyperelliptic curves, an implementation of 2-descent is readily available in Magma \cite{magma}. Moreover, once one rational non-torsion point on the Jacobian is known, a basis of annihilating differentials is easily computed using the algorithm of \cite{Balakrishnan}. Finally, the particularly nice form of the basis of $H^0(C, \Omega^1)$ makes the analysis of a certain divisor relevant to the method more amenable to direct computation.\end{remark}

\section*{Acknowledgements} We are grateful to Andrew Sutherland for isomorphism testing a large collection of genus 3 hyperelliptic curves. We also thank Steffen M\"uller, Riya Parankimamvila Mamachan, and Hector Pasten for useful discussions and for carefully reading a previous version of this manuscript.
The first author was supported by NSF grant DMS-1945452 and Simons Foundation
grants \#550023 and \#1036361, and the second author was supported by  Simons Foundation grant \#550023.


\section{A Chabauty--Coleman bound for unexpected quadratic points of curves}
In this section, we review a special case of the Caro--Pasten method \cite{CaroPasten2021}, which gives an upper bound for the number of rational points on surfaces inside an abelian variety with Mordell-Weil rank $1$. Additionally, we give a sketch of Siksek's work on symmetric Chabauty \cite{Siksek} in the special case of the symmetric square of a curve.

\subsection{A Chabauty-Coleman bound for surfaces}
Let $C$ be a smooth, projective, and geometrically irreducible curve over the rationals of genus $3$. Let $J$ denote the Jacobian of $C$ and assume that the Mordell-Weil rank of $J$ is $1$. 

For each point $P$ in $C$, we have the following diagram, which represents a general construction of a morphism from the symmetric square $C^{(2)}$ of $C$ to $J$:
\begin{equation}\label{construction}
\xymatrixcolsep{5pc}\xymatrix{
C\ar[r]^{j_P}\ar[drr]_{i_P} &C^2\ar[r]^{\pi} &C^{(2)}\ar[d]^{\psi_P}\\
&&J,
}
\end{equation}
where 
\begin{align*}
j_P(Q)&=(Q,P),\\
i_P(Q)&=[(Q)-(P)],    
\end{align*} 
$\pi$ is the canonical quotient morphism, and 
\begin{equation*}
\psi_P([(Q,R)])=[(Q)+(R)-2(P)]    
\end{equation*}
for all $Q,R\in C$. The image of $C^{(2)}$ in $J$ is a closed subvariety $W_2$ of $J$. It can be expressed as 
\begin{equation*}
W_2:=i_P(C)+i_P(C).
\end{equation*}
Caro and Pasten \cite[Corollary 1.15]{CaroPasten2021} provide an upper bound for the number of rational points in $W_2$.

\begin{theorem}[Method of Caro--Pasten]\label{CoroQuadratic3} Let $C$ be a smooth, geometrically irreducible, projective curve over $\Q$ of genus $3$ such that its Jacobian $J$ has Mordell-Weil rank at most $1$. Let  $p\ge 521$ be a prime of good reduction for $C$. Suppose that $W_2$ does not contain elliptic curves over $\F_p^{\rm alg}$. Then  $W_2(\Q)$  is finite and
\begin{equation}
\# W_2(\Q)\le \# W_2(\F_p) +  6\cdot \frac{p-1}{p-2}\cdot \left(p+4p^{1/2} + 5\right).
\end{equation}
\end{theorem}

\begin{remark} The condition $p \geq 521$ comes from the geometric bound $p > (128/9)c^2_1(W_2)^2$, where $c^2_1(W_2)$ is the degree of the first Chern class, or equivalently, the self-intersection of the canonical class. This geometric bound is, in particular, used in the work of Caro and Pasten to prove that a certain $2$-form in $H^0(W_2,\Omega^2_{W_2/\F_{p}^{alg}})$ is not the zero section (cf. \cite[Lemma 5.7]{CaroPasten2021}).
\end{remark}

Caro--Pasten's upper bound is, in fact, an upper bound for $\#W_2(\Q_p)\cap\overline{J(\Q)}$, where $\overline{J(\Q)}$ denotes the $p$-adic closure of $J(\Q)$ in $J(\Q_p)$. They establish this bound locally, i.e., they bound $\#W_2(\Q_p)\cap \overline{J(\Q)}\cap U_x$ for every $x\in W_2(\F_p)$, where $U_x$ is the residue disc associated with $x$, using overdetermined systems of differential equations in positive characteristic ($\omega$-integral curves). 

To begin with, we parametrize $\overline{J(\Q)}\cap U_x$ with a $p$-adic analytic map $\gamma \colon  p\Z_p\to U_x\subseteq J(\Q_p)$. Let $f$ be a local equation for $W_2$ in $U_x$. Then $\#W_2(\Q_p)\cap\overline{J(\Q)}\cap U_x$ is bounded by the number of zeros of the $p$-adic analytic function 
\begin{equation}\label{function h}
h\coloneqq f\circ \gamma=\sum_{n=1}^{\infty}c_nz^n
\end{equation}
on $p\Z_p$. By \cite[Lemma 6.2]{CaroPasten2021}, it is enough to find some small $N$ such that $|c_N|$ is not too small, say $|c_N|\ge 1$. To find this small $N$ (depending on $x$), we introduce the following definitions:

\begin{definition}[$\omega$-integrality]\label{w-integral}
Let $k$ be a field, let $\phi\colon X\to Y$ be a morphism of $k$-schemes, and let $\omega\in H^0(Y,\Omega^1_{Y/k})$. We say that $\phi$ is $\omega$-integral if the composition
\begin{equation}\label{phi bullet}
\phi^\bullet \colon H^0(Y,\Omega^1_{Y/k})\to H^0(Y,\phi_*\phi^*\Omega^1_{Y/k})=H^0(X,\phi^*\Omega^1_{Y/k})\to H^0(X,\Omega^1_{X/k})
\end{equation}
satisfies $\phi^\bullet(\omega)=0$.  
\end{definition}

\begin{definition}
Let $\Ann(p, J(\Q))$ be the $2$-dimensional $\Q_p$-vector space defined as
\begin{equation*}
\Ann(p,J(\Q)) \coloneqq \{\omega \in H(J_{\Q_p},\Omega^1)\colon  \int_{0}^{R}\omega =0
\text{ for all }R\in J(\Q)\},
\end{equation*}
where $\int_{0}^{R}\omega$ denotes the associated Coleman integral (see \cite{Wetherell97} for a definition).
\end{definition}

Let $k \coloneqq \F_p^{alg}$, let  
\begin{equation*}
V_m^k \coloneqq \Spec(k[z]/(z^{m+1})),   
\end{equation*}
let $\xi$ denote the generic point of $V_m^k$, and denote by $(W_2)_{k}$ the surface $W_2\otimes k$.

\begin{proposition}\label{nonzero wedge} Let $p \geq 521$ be a prime of good reduction for $C$. There exists a basis  \{$\omega_1,\omega_2$\} of $\Ann(p,J(\Q))$ such that $$w_1\wedge w_2\in H^0((W_2)_{k}, \Omega^2_{(W_2)_{k}/k})$$ is not the zero section, where $w_1,w_2\in H^0((W_2)_{k},\Omega^1_{(W_2)_{k}/k})$ are obtained by reducing $\omega_1,\omega_2$ mod $p$ and restricting to $(W_2)_{k}$.
\end{proposition}

\begin{proof}Under the assumptions on $p$, this basis exists by \cite[Lemma 9.8]{CaroPasten2021}.\end{proof}

Suppose there is $m < p$, such that $h$ in \eqref{function h} satisfies $|c_i| < 1$ for every $i \leq m$.  Then there exists a closed immersion $\phi_m$ defined over $k$:
\begin{equation*}
\phi_m\colon V_m^k\to (W_2)_{k},
\end{equation*}
at $x\in W_2(\F_p)$ which is $w_i$-integral for $i=1,2$ (see \cite[Section 8]{CaroPasten2021}).  Let $m(x)$ denote the upper bound for any $m$ such that there is a closed immersion $\phi_m$ as above. 
Let $D=\divisor(w_1\wedge w_2)$ and let us write $D=a_1C_1+\cdots +a_\ell C_\ell$ where $a_j$ are positive integers and $C_j$ are irreducible curves for each $j$. For each $j$, let $\nu_j\colon \widetilde{C}_j\to (W_2)_{k}$ be the normalization map of $C_j$ composed with the inclusion $C_j\to (W_2)_{k}$. To establish this upper bound $m(x)$, we use the following theorem (cf. \cite[Theorem 4.4]{CaroPasten2021}):
\begin{theorem} \label{ThmOver} Let $m\ge 0$ be an integer such that there is a closed immersion $\phi\colon  V_m^k\to (W_2)_{k}$ supported at $x$ (i.e.~with $\phi(\xi)=x$) which is $w_i$-integral for both $i=1$ and $i=2$. Then for every $w_0\in H^0((W_2)_{k},\Omega^1_{(W_2)_{k}/k})$ of the form $w_0=c_1w_1+c_2w_2$ with $c_1,c_2\in k$, we have
\begin{equation}\label{EqnOver}
m\le \sum_{j=1}^\ell \sum_{y\in \nu_j^{-1}(x)} a_j\cdot \left(\ord_y (\nu_j^\bullet (w_0))+1 \right),
\end{equation}
where $\nu_j^\bullet$ is as in \eqref{phi bullet} from Definition \ref{w-integral}. 
\end{theorem}

Finally, according to \cite[Proposition 9.14]{CaroPasten2021}, we have 
\begin{equation}\label{prop914}
\#W_2(\Q_p)\cap\overline{J(\Q)}\cap U_x\leq \frac{p-1}{p-2}m(x)+1.    
\end{equation}
Consequently
\begin{align*}
\#W_2(\Q_p)\cap\overline{J(\Q)}&\leq\sum_{x\in W_2(\F_p)}\#W_2(\Q_p)\cap\overline{J(\Q)}\cap U_x\nonumber\\
&\leq \#W_2(\F_p)+\frac{p-1}{p-2}\sum_{x\in W_2(\F_p)}m(x).
\end{align*}

\begin{observation}\label{sharpness m}
We have that \cite[Proposition 9.14]{CaroPasten2021} implies that
\begin{equation*}
\#W_2(\Q_p)\cap\overline{J(\Q)}\cap U_x\leq \left\lfloor\frac{p-1}{p-2}m(x)\right\rfloor+1=m(x)+1+\left\lfloor\frac{m(x)}{p-2}\right\rfloor.
\end{equation*}
In particular, $\#W_2(\Q_p)\cap\overline{J(\Q)}\cap U_x\leq m(x)+1$ whenever $m(x)\leq p-3$. Furthermore, $\#W_2(\Q_p)\cap\overline{J(\Q)}\cap U_x\leq m(x)+2$ whenever $p-2\leq m(x)< p$.
\end{observation}

\begin{remark}\label{condition reduction}
The condition that $(W_2)_{k}$ does not contain elliptic curves over $\F_p^{alg}$ ensures that the geometric genus of each $C_j$ in the definition of $D$ used in \eqref{EqnOver} is greater than $1$. Otherwise, the right-hand side of \eqref{EqnOver} will always be infinite.
\end{remark} 

\subsection{Siksek's Criterion}\label{Comparison} As the $d=2$ case of his work in \cite{Siksek}, Siksek provided an explicit and practical criterion to show that a given subset $\mathcal{L}\subset W_2(\Q)$ is, in fact, equal to $W_2(\Q)$. His method can be summarized in the following three steps:
\begin{itemize}
    \item [(a)] Fix a prime $p$ such that the points in $\mathcal{L}$ lie in different residue disks. 
    \item [(b)] Using \cite[Theorem 3.2]{Siksek} and \cite[Theorem 4.3]{Siksek}, demonstrate that no point $x\in W_2(\Q)\setminus\mathcal{L}$ shares the same residue disk as any point of $\mathcal{L}$.
    \item [(c)] By analyzing the reduction of $J(\Q)$ modulo several primes, show that the reductions  of $W_2(\Q)$ and $\mathcal{L}$ modulo $p$ coincide, thereby implying $W_2(\Q)=\mathcal{L}$.
\end{itemize} 

We now state specific instances of \cite[Theorem 3.2]{Siksek} and \cite[Theorem 4.3]{Siksek} that we will use here. The latter theorem will be used to handle the residue disk associated with $0_J$ and applied in the examples.

\begin{theorem}[{\cite[Theorem 3.2]{Siksek}}]\label{Siksek method}
Let $C$ be a curve, let $p\ge 5$ be a prime of good reduction for $C$, and let $\omega_1,\omega_2$ be a basis for $\Ann(p,J(\Q))$. Let $\mathcal{D}=[(P_1,P_2)]\in C^{(2)}(\Q)$ and let $t_j$ be a uniformizer at $P_j$ that is also a uniformizer modulo $p$. We define the matrix $A(P_1,P_2)$ as  
\[
\begin{pmatrix}
\left.\frac{\omega_1}{dt_1}\right|_{t_1=0} & \left.\frac{\omega_2}{dt_1}\right|_{t_1=0}\\
\left.\frac{\omega_1}{dt_2}\right|_{t_2=0}& \left.\frac{\omega_2}{dt_2}\right|_{t_2=0}
\end{pmatrix}\qquad\text{or}\qquad 
\begin{pmatrix}
\left.\frac{\omega_1}{dt_1}\right|_{t_1=0} & \left.\left(\frac{\omega_1}{dt_1}\right)'\right|_{t_1=0}\\
\left.\frac{\omega_2}{dt_1}\right|_{t_1=0} & \left.\left(\frac{\omega_2}{dt_1}\right)'\right|_{t_1=0}
\end{pmatrix},
\]
depending on whether $P_1\neq P_2$ or $P_1= P_2$, respectively. Note that $\omega_i/dt_1$  is a power series in $t_1$, which Siksek gives explicitly.
If  
$\det(A(P_1,P_2))$
\textbf{}is nonzero modulo $p$, no unexpected quadratic points in the residue disk are associated with $\mathcal{D}$.
\end{theorem}

\begin{theorem}[{\cite[Theorem 4.3]{Siksek}}]\label{residue 0J}
Let $C$ be a hyperelliptic curve, let $p\ge 5$ be a prime of good reduction for $C$, let $\omega_1,\omega_2$ be a basis for $\Ann(p,J(\Q))$, let $\mathcal{D}=[(x,y),(x,-y)]\in C^{(2)}(\Q)$, and let $t$ be a uniformizer at $(x,y)$ that is also a uniformizer modulo $p$.

If  
$\left.\frac{\omega_i}{dt}\right|_{t=0}$
is nonzero modulo $p$ for $i=1$ or $i=2$, no unexpected quadratic points in the residue disk are associated with $\mathcal{D}$.
\end{theorem}


\section{Explicit computations}

Now, we restrict to the following setup. Let $C$ be a hyperelliptic curve of genus $3$ over $\Q$ defined by an equation $y^2 = f(x)$, where $\deg f(x) = 7$. Let $\iota$ denote the hyperelliptic involution, and let $\infty$ denote the unique point at infinity on $C$. Let us denote by $\psi, i, j,$ the morphisms $\psi_\infty,i_\infty,j_\infty$, respectively, in diagram \eqref{construction}. 

\subsection{Classification of \texorpdfstring{$\divisor(w_1\wedge w_2)$}{TEXT}}

In this subsection, we fix a prime $p\geq 5$ of good reduction for $C$. As before, let $k$ denote $\F_p^{\rm alg}$.
Let $\omega_1$ and $\omega_2$ be two basis elements of $\Ann(p, J(\Q))$ such that $w_1$ and $w_2$, their respective reductions mod $p$,   are nonzero. (Note that such a basis exists for $p \geq 5$ and appropriate $p$-adic precision, as in \cite{Balakrishnan}. In particular, we do not yet need to assume that $p \geq 521$.)

For $n=1,2$, let us rewrite $\omega_n$ in the canonical basis of $H^0(J,\Omega^1_{J/\Q_p})$ as follows:
\begin{equation*}
\omega_n=\alpha_{n0}i_{*}\frac{dx}{y}+\alpha_{n1}i_{*}\frac{xdx}{y}+\alpha_{n2}i_{*}\frac{x^2dx}{y},
\end{equation*}
where $\alpha_{nm}\in \Z_p$. Consequently, we have the following expression:
\begin{equation}\label{wedge equation}
w_1\wedge w_2=\beta_{01} i_{*}\frac{dx}{y}\wedge i^{*}\frac{xdx}{y}+\beta_{02} i^{*}\frac{dx}{y}\wedge i_{*}\frac{x^2dx}{y}+\beta_{12} i^{*}\frac{xdx}{y}\wedge i_{*}\frac{x^2dx}{y},
\end{equation}
where $\beta_{nm}$ is the reduction modulo $p$ of $\alpha_{1n}\alpha_{2m}-\alpha_{1m}\alpha_{2n}$. The following lemma allows us to express $w_1\wedge w_2$ restricted to $W_2$ in terms of local parameters in $C^2$.

\begin{lemma} \label{lemma divisor} Let $x_i,y_i$ be local coordinates of the $i$-th component of $C^2$. Then we have
\begin{equation}
\psi^{*}(w_1\wedge w_2)=(\beta_{01}+\beta_{02}(x_1+x_2)+\beta_{12}x_1x_2)(x_2-x_1)\frac{dx_1\wedge dx_2}{y_1y_2}.
\end{equation}
\end{lemma}

\begin{proof}
By \cite[Proposition 5.3]{MilneAV}, the morphism $\psi^{*}i_{*}\colon \Omega_C\to \Omega_{C^{(2)}}$ is defined by $\psi^{*}i_{*}\omega=\omega_{\pi_{1}(C^2)}+\omega_{\pi_{2}(C^2)}$. Therefore we obtain
\begin{equation*}
\psi^{*}i_{*}\frac{x^{n}dx}{y}=\frac{x_1^{n}dx_1}{y_1}+\frac{x_2^{n}dx_2}{y_2}
\end{equation*}
for $n=0,1,2$. Applying $\psi^{*}$ to \eqref{wedge equation} yields the desired result.
\end{proof}

\begin{remark}\label{D +Siksek}
After a similar analysis to that in Lemma \ref{lemma divisor} of the $2$-form $w_1 \wedge w_2$, one can show the following result: for any genus $3$ curve and any $[(P_1, P_2)] \in C^{(2)}(\Q)$, the matrix $A(P_1, P_2)$ in Theorem \ref{Siksek method} is invertible modulo $p$ if and only if the point $[(P_1) + (P_2) - 2(\infty)]$ is not in the support of $D$ defined in Theorem \ref{ThmOver}.

In what follows, we use a combination of Theorem \ref{residue 0J} and the Caro--Pasten method, exploiting the upper bound for $x$ contained in the support of $D$.
\end{remark}

Now, considering Theorem \ref{ThmOver}, our goal is to compute the divisor of $w_1\wedge w_2$ restricted to $W_2$. Recall that $C^{(2)}$ is the blow-up at the origin of $W_2$ with the exceptional divisor $\nabla$, which is
the push-forward under $\pi:C^2\to C^{(2)}$ of the divisor $\{
(P,\iota(P))\colon P\in C\}$ (see \cite[pp. 739]{McCrory1992} for a detailed proof).  Thus for every $\eta\in H^0(J,\Omega^{1}_{J/k})$ we have
\begin{equation*}
\divisor \left(\eta\Big|_{W_2}\right)=\psi_{*}(\divisor(\psi^{*}\eta)-\nabla).
\end{equation*}
Hence, we need to compute $\divisor(\psi^{*}(\omega_1\wedge\omega_2))\in \Div(C^{(2)})$ in Lemma \ref{lemma divisor}.

\begin{lemma}\label{Divisor D}
The Weil divisor $D$ on $W_2$ associated with $w_1\wedge w_2$ satisfies one of the following conditions: 
\begin{itemize}
\item[(i)] $D=C_P+C_{\iota(P)}$ for some $P\in C(k)$, where, for $P\in C(k)$, $C_P\subset J$ denotes the curve $i(C)+[(P)-(\infty))]$.
\item[(ii)] $D=\psi_{*}\pi_{*}Z$, where $Z=\mathbb{V}(\beta_{01}+\beta_{02}(x_1+x_2)+\beta_{12}x_1x_2)\subset C^2$.
\end{itemize}
\end{lemma}

\begin{proof} We compute the divisor associated to $\psi^{*}(w_1\wedge w_2)$, from which we deduce the result for $w_1 \wedge w_2$.  Suppose first that there exist $a,b\in k$ such that
\begin{equation}\label{divisor constant}
\psi^{*}(w_1\wedge w_2)=(b x_1+a)(b x_2+a)(x_2-x_1)\frac{dx_1\wedge dx_2}{y_1y_2}.
\end{equation}
In this case, the divisor associated with this $2$-form on $C^{(2)}$ is 
\begin{equation*}
C_{(Q)}+C_{(\iota(Q))} + \nabla,
\end{equation*}
where $Q=(-a/b,\sqrt{f(-a/b)})$, $C_{(Q)}=\{[(R,Q)]\in C^{(2)}\colon R\in C\}$, and $\iota$ denotes the involution on $C^{(2)}$. 

Now let $Z$ be the Zariski closed subset of $C^{2}$ defined by the local equation $\beta_{01}+\beta_{02}(x_1+x_2)+\beta_{12}x_1x_2=0$. Let $Z'$ denote the push-forward of $Z$ via $\pi$. Note that when $\psi^{*}(w_1\wedge w_2)$ is not of the form \eqref{divisor constant}, then the divisor associated with this $2$-form on $C^{(2)}$ is
\begin{equation*}
\divisor(\psi^{*}(w_1\wedge w_2))=Z'+\nabla,
\end{equation*}
which yields the desired result. 
\end{proof}

\begin{observation}\label{Z nonconstant}
Let $ Z_1 $ be an irreducible component of $ Z = \mathbb{V}(\beta_{01} + \beta_{02}(x_1 + x_2) + \beta_{12}x_1x_2) \subset C^2 $. Let $ \varphi: Z \to C $ be the morphism associated with the projection onto the first component of $ C^2 $. Assume that the restriction of $ \varphi $ to $ Z_1 $ is constant. By symmetry, there exist $ a, b \in k $ such that 
\[
\beta_{01} + \beta_{02}(x_1 + x_2) + \beta_{12}x_1x_2 = (ax_1 + b)(ax_2 + b).
\]
In this case, $ \psi^*(w_1 \wedge w_2) $ takes the form given in equation \eqref{divisor constant}.

Now, assume that the restriction of $ \varphi $ to each irreducible component of $ Z $ is surjective. By the previous observation, this is equivalent to $ \psi^*(w_1 \wedge w_2) $ not being in the form given by equation \eqref{divisor constant}. In this case, if we fix $ (x_2, y_2) $, the equation defining $ Z $ also fixes $ x_1 $; hence, $ \varphi $ is generically two-to-one. Consequently, $ Z $ has at most two irreducible components.
\end{observation}

The following lemma classifies when $Z$ in Lemma \ref{Divisor D}(ii) is reducible.

\begin{lemma}\label{Z reducible}
Suppose that $\psi^{*}(w_1\wedge w_2)$ is not of the form given by \eqref{divisor constant}. Then $Z$ is reducible precisely in one of the following scenarios:
\begin{itemize}
\item [(i)]  $\beta_{12}=0$ and the function $x\mapsto -(x+\beta_{01}/\beta_{02})$ permutes the roots of $f$.
\item[(ii)]  $\beta_{12}\neq 0$, $-\beta_{02}/\beta_{12}$ is a root of $f$, and the function $x\mapsto (-\beta_{02}x-\beta_{01})/(\beta_{12}x+\beta_{02})$ permutes the other roots of $f$.
\end{itemize}
\end{lemma}

\begin{proof}
First note that  $Z$ has at most two irreducible components, by Observation \ref{Z nonconstant}.

Let $ \varphi: Z \to C $ be the morphism associated with the projection onto the first component of $ C^2 $. Suppose that $Z$ is reducible, i.e., $Z=Z_1\cup Z_2$.
Therefore, $Z$ is singular at every point in $Z_1\cap Z_2$, precisely the set of branched points of $\varphi$.

Note that $P=((a_1,b_1),(a_2,b_2))\in Z$ is a branched point of $\varphi$ if $b_2=0$, and computing the Jacobian matrix implies that $Z$ is singular at $P$ if $b_1=b_2=0$. Consequently, when $Z$ is reducible and  $\xi\neq-\frac{\beta_{02}}{\beta_{12}}$ is also a root of $f$, we have that $(-\beta_{02}\xi-\beta_{01})/(\beta_{12}\xi+\beta_{02})$ is a root of $f$. Therefore, the function $x\mapsto (-\beta_{02}x-\beta_{01})/(\beta_{12}x+\beta_{02})$ permutes the roots of $f$ that are different from $-\frac{\beta_{02}}{\beta_{12}}$. 

Now note that when $\beta_{12}=0$, we have that $\{(\infty,\infty)\}=Z\setminus \mathbb{V}(\beta_{01}+\beta_{02}(x_1+x_2)+\beta_{12}x_1x_2)$ and this point is always a branched point of $\varphi$, and $Z$ is singular at $P$. On the other hand, when $\beta_{12}\neq0$, the points of $Z\setminus \mathbb{V}(\beta_{01}+\beta_{02}(x_1+x_2)+\beta_{12}x_1x_2)$ are of the form $(\infty,(-\frac{\beta_{02}}{\beta_{12}}, \pm b))$ and $((-\frac{\beta_{02}}{\beta_{12}}, \pm b),\infty)$. In this case, $((-\frac{\beta_{02}}{\beta_{12}}, \pm b),\infty)$ are branched points of $\varphi$. Computing the Jacobian matrix, one gets that $Z$ is singular at these points precisely when $b=0$, i.e., $-\frac{\beta_{02}}{\beta_{12}}$ is a root of $f$. Consequently, if $Z$ is reducible, then (i) or (ii) is satisfied. 

Assume that (i) or (ii) is satisfied. In either of these cases, we have that the polynomial
\begin{equation*}
f\left(-\frac{\beta_{02}x+\beta_{01}}{\beta_{12}x+\beta_{02}}\right)(\beta_{12}x+\beta_{02})^{8}
\end{equation*}
has degree $7$ and the same roots as $f$. Therefore, there is $\gamma\in \F_p^{*}$ such that 
\begin{equation}\label{eq polynomials}
f(x)=\gamma f\left(-\frac{\beta_{02}x+\beta_{01}}{\beta_{12}x+\beta_{02}}\right)(\beta_{12}x+\beta_{02})^{8}.  
\end{equation}
This implies that $y_1^2=\gamma y_2^2(\beta_{12}x_1+\beta_{02})^{8}$. Therefore, we have that 
\begin{equation}\label{reduction}
Z=\mathbb{V}_Z(y_1-\sqrt{\gamma}y_2(\beta_{12}x_1+\beta_{02})^{4})\cup  \mathbb{V}_Z(y_1+\sqrt{\gamma}y_2(\beta_{12}x_1+\beta_{02})^{4}),    
\end{equation}
implying that $Z$ is reducible.
\end{proof}

Let us assume that $Z=Z_1\cup Z_2$ is reducible. By Lemma \ref{Divisor D}, $D=\psi_{*}\pi_{*}Z$ has at most two irreducible components. The following lemma implies that we can apply the method of Caro--Pasten (Theorem \ref{CoroQuadratic3}) unless $D$ has two different irreducible components, which occurs precisely when $\pi(Z_1)\neq \pi(Z_2)$. Furthermore, we can characterize when this happens.

\begin{lemma}\label{D reducible}
Suppose that $D$ has two irreducible components, i.e., $D=D_1+ D_2$, where $D_1$ and $D_2$ are different irreducible curves. Then the scenario of Lemma \ref{Z reducible}(ii) occurs and 
\begin{equation}
f'(-\beta_{02}/\beta_{12})=a_7(\beta_{02}^2-\beta_{12}\beta_{01})^3,
\end{equation}
where $a_7$ is the leading coefficient of $f$. In this case, there is an irreducible component of $D$ with geometric genus $1$.
\end{lemma}

\begin{proof}
By Lemma \ref{Divisor D}, $D=\psi_{*}\pi_{*}Z$. Note that $D$ has two distinct irreducible components if and only if $Z=Z_1\cup Z_2$
and 
\begin{equation}\label{Z1neqZ2}
\psi_{*}\pi_{*}Z_1\neq \psi_{*}\pi_{*}Z_2.    
\end{equation}
Since $Z$ is invariant by the action of $S_2=\langle\sigma\rangle$, then $S_2$ acts on its irreducible components. Thus,  \eqref{Z1neqZ2} is equivalent to 
having $\sigma(Z_i)= Z_i$ for $i=1,2$. By \eqref{reduction}, this occurs precisely when $$\gamma(\beta_{02}^2-\beta_{12}\beta_{01})^4=1.$$ 
First, assume that $\beta_{12}=0$ and, without loss of generality, $\beta_{02}=1$. Since $f$ has odd degree, \eqref{eq polynomials} implies that $\gamma=-1$, and so $\gamma(\beta_{02}^2-\beta_{12}\beta_{01})^4\neq 1$. Consequently, $\sigma(Z_1)=Z_2$ and vice versa, which implies $D$ is irreducible.

Let us assume that $\beta_{12}\neq0$. By Lemma \ref{Z reducible}, we have that $-\beta_{02}/\beta_{12}$ is a root of $f$, then $f'(-\beta_{02}/\beta_{12})\neq 0$. Evaluating the derivative of \eqref{eq polynomials} at $-\beta_{02}/\beta_{12}$ we have that
$$
f'(-\beta_{02}/\beta_{12})=\gamma a_7(\beta_{02}^2-\beta_{12}\beta_{01})^7.
$$
Therefore, when $f'(-\beta_{02}/\beta_{12})=a_7(\beta_{02}^2-\beta_{12}\beta_{01})^3$ we obtain that $\gamma(\beta_{02}^2-\beta_{12}\beta_{01})^4=1$ and this proves the first statement of the lemma.

Since $Z$ is reducible, $\varphi:Z_i\to C$ is an isomorphism for $i=1,2$. Thus, $Z_i$ is non-singular and has genus $g(Z_i)=3$. Since $\psi^{*}(w_1\wedge w_2)$ is not of the form given by \eqref{divisor constant}, the polynomial 
\begin{equation*}
\beta_{01}+2\beta_{02}x+\beta_{12}x^2
\end{equation*}
is not constant. Therefore, $Z$ intersects the diagonal $\Delta$ of $C^2$, and without loss of generality, we assume that $Z_1$ intersects $\Delta$. Let $Z_1'$ be the image of $Z_1$ via $\pi$. Under the assumption that $Z_1$ is invariant under the action of $S_2$, there is a morphism $Z_1\to Z_1'$ of degree $2$  that ramifies because $Z_1\cap\Delta\neq\emptyset$. By \cite[Proposition II.7.14]{Hartshorne}, we have a morphism from $Z_1$ to the normalization $\widetilde{Z_1'}$ of $Z_1'$ which has degree $2$ and is ramified. Since $\ch k\neq 2$, this morphism has only tame ramification. Therefore, Riemann--Hurwitz (see IV.2.4 \textit{loc. cit.}) implies that $g(\widetilde{Z_1'})=1$. Since $C^{(2)}$ is the blow-up at the origin of $W_2$, we have that $\widetilde{Z_1'}$ is the normalization of a connected component of $D$.
\end{proof}

\begin{remark}\label{Z' normalization}
Let $Z'$ be the image of $Z$ via $\pi$ and assume that $Z'$ is irreducible. Since $C^{(2)}$ is the blow-up at the origin of $W_2$, the normalization of $Z'$ is also the normalization of $D$.  We record these maps in the diagram below:
\begin{equation}\label{diagram W2}
\begin{tikzcd}[row sep=2pc, column sep=6pc]
    & \widetilde{Z} \arrow[ddl, "{\widehat{\varphi}}"'] \arrow[d, "{\nu}"] \arrow[r, "{\widetilde{\pi}}"] & \widetilde{Z'} \arrow[d, "{\nu'}"] \arrow{dr}{\widetilde{\psi}} & \\
    & Z \arrow[dl, "{\varphi}"'] \arrow[hookrightarrow]{d} \arrow[r, "{\pi}"] & Z' \arrow[hookrightarrow]{d} \arrow[r, "{\psi}"] & D \arrow[hookrightarrow]{d} \\
    C & C^2 \arrow{l}[swap]{\pi_1} \arrow[r, "{\pi}"] & C^{(2)} \arrow[r, "{\psi}"] & W_2
\end{tikzcd}
\end{equation}
\end{remark}

\begin{observation}
As stated in Remark \ref{condition reduction}, we require the geometric genus of each irreducible component to be greater than $1$. By Lemmas \ref{Divisor D} and \ref{D reducible}, we only exclude the case when $Z$ is reducible, and every component is invariant under the action of $S_2$.
\end{observation}

\subsection{Bounding \texorpdfstring{$p$ and $m(x)$}{TEXT}}
In this subsection, we give a lower bound on the prime $p$ for which the method applies and an upper bound on the size of $W_2(\Q)$. Let $Z\subset C^2$ denote the Zariski closed set defined by $\beta_{01}+\beta_{02}(x_1+x_2)+\beta_{12}x_1x_2$. We will work separately in the following three different cases:

\begin{itemize}
    \item \textbf{Case I:} $Z$ is reducible and there are $a,b\in \F_p^{\rm alg}$ such that 
    \[
    \beta_{01}+\beta_{02}(x_1+x_2)+\beta_{12}x_1x_2=(b x_1+a)(b x_2+a).
    \]
    \item \textbf{Case II:} $Z$ is reducible, not in the previous case, and 
    \[
    f'(-\beta_{02}/\beta_{12})\neq a_7(\beta_{02}^2-\beta_{12}\beta_{01})^3.
    \]
    \item \textbf{Case III:} $Z$ is irreducible.
\end{itemize}

\noindent We begin with Case I. 
\begin{proposition}[Case I]\label{bound p case 1}
Suppose that there exist $\alpha,\beta\in \F_p$ such that
\begin{equation*}
\psi^{*}(w_1\wedge w_2)=(\beta x_1+\alpha)(\beta x_2+\alpha)(x_2-x_1)\frac{dx_1\wedge dx_2}{y_1y_2}.
\end{equation*}
Then, for every $x\in W_2(\F_p)$ we have $m(x)\leq 6$. Consequently, we can apply Theorem \ref{ThmOver} with any prime of good reduction for $C$ such that $p\geq 11$ and we obtain
\begin{equation}
\#W_2(\Q)\leq \#W_2(\F_p)+2\# C(\F_p)+4.
\end{equation}
\end{proposition}

\begin{proof} 
By Lemma \ref{Divisor D}(a) we have that $D=C_{P}+C_{\iota(P)}$, where $P=(-\alpha/\beta,\sqrt{f(-\alpha/\beta)})$. In this particular case, we can assume that $w_1=(\beta x+\alpha )dx/y$ and $w_2=(\beta x^2+\alpha x)dx/y$ and by the definition of $i$, we have
\begin{equation*}
\ord_{[(R)+(P)-2(\infty)]}(i_{*}w|_{C_P})=\ord_{R}(w),    
\end{equation*}
for any $w\in\Span_{\F_p}\{w_1,w_2\}$. The divisors of $w_1$ and $w_2$ are the following:
\begin{align*}
\divisor(w_1)&=(0,\sqrt{f(0)})+(0,-\sqrt{f(0)})+(-\alpha/\beta,\sqrt{f(-\alpha/\beta)})+(-\alpha/\beta,-\sqrt{f(-\alpha/\beta)}),\\
\divisor(w_2)&=2(\infty)+(-\alpha/\beta,\sqrt{f(-\alpha/\beta)})+(-\alpha/\beta,-\sqrt{f(-\alpha/\beta)}).
\end{align*}
Therefore, when $-\alpha/\beta$ is a root of $f$, we have that $D=2C_P$, and Theorem \ref{ThmOver} implies that
\begin{itemize}
    \item $m(x)=0$ for $x\in (W_2\setminus C_P)(\F_p)$,
    \item $m(x)\leq 2$ for $x\in C_P(\F_p)\setminus \{[2(P)-2(\infty)]\}$, and 
    \item $m([2(P)-2(\infty)])\leq 6$ since $\ord_{[2(P)-2(\infty)]}(i_{*}w|_{C_P})=2$.
\end{itemize}

Applying \cite[Proposition 9.14]{CaroPasten2021} we obtain
\begin{align*}
\#W_2(\Q)&=\sum_{x\in W_2(\F_p)}\#W_2(\Q_p)\cap\overline{J(\Q)}\cap U_x\leq \sum_{x\in W_2(\F_p)}\left\lfloor \frac{p-1}{p-2}m(x)\right\rfloor +1\nonumber\\
&\leq \#W_2(\F_p)+2\# C(\F_p)+4,    
\end{align*}
where the last inequality comes from Observation \ref{sharpness m}, since $m(x)\leq p-3$.

On the other hand, 
when $-\alpha/\beta$ is not a root of $f$, Theorem \ref{ThmOver} implies that
\begin{itemize}
    \item $m(x)=0$ for $x\in (W_2\setminus (C_P\cup C_{\iota(P)}))(\F_p)$,
    \item $m([2(P)-2(\infty)]), m([2(\iota(P))-2(\infty)])\leq 2$ and $m([(P)+(\iota(P))-2(\infty)])\leq 4$, and 
    \item $m(x)\leq1$ for $x\in (C_P\cup C_{\iota(P)})(\F_p)$ different from the three points in the previous item.
\end{itemize}

Applying \cite[Proposition 9.14]{CaroPasten2021} we obtain
\begin{align*}
\#W_2(\Q)&=\sum_{x\in W_2(\F_p)}\#W_2(\Q_p)\cap\overline{J(\Q)}\cap U_x\leq \sum_{x\in W_2(\F_p)}\left\lfloor \frac{p-1}{p-2}m(x)\right\rfloor +1\nonumber\\
&\leq \#W_2(\F_p)+2\# C(\F_p)+2,    
\end{align*}
which yields the desired result.
\end{proof}

We follow the same strategy as above, now in Case II:
\begin{proposition}[Case II]\label{bound p case 2}
Suppose that $Z=\mathbb{V}(\beta_{01}+\beta_{02}(x_1+x_2)+\beta_{12}x_1x_2)\subset C^{2}$ is reducible and its image via $\pi$ is irreducible. Then, for every $x\in W_2(\F_p)$ we have $m(x)\leq 4$. Consequently, we can apply Theorem \ref{ThmOver} with any prime of good reduction for $C$ such that $p\geq 7$ and we obtain the following upper bound for $W_2(\Q)$:
\begin{equation}
\#W_2(\Q)\leq \#W_2(\F_p)+2\#C(\F_p)+4.
\end{equation}
\end{proposition}

\begin{proof}
By hypothesis, we have that $Z=Z_1\cup Z_2$ and $\pi(Z_1)=\pi(Z_2)$. Then, $\pi$ defines an isomorphism from $Z_i$ to $Z'=\pi(Z)$. Notice that $\pi_{*}(Z)=2\cdot Z'$. Therefore, there exists an isomorphism $\varphi\circ\pi^{-1}:Z'\to C$, as in \eqref{diagram W2}, which implies that $Z'$ is a non-singular curve, and by Remark \ref{Z' normalization}, it is the normalization of $D$. Since $0_J$ is the unique singular point of $D$, the corresponding $\psi$ in \eqref{diagram W2} is one-to-one outside $Z'\cap \nabla$.

Since $\psi^{*}(w_1\wedge w_2)$ is not of the form of \eqref{divisor constant}, the supports of $w_1$ and $w_2$ are disjoint. Then for every $x\in D(\F_p)\setminus\{0_J\}$ we have $\ord_{\psi^{-1}(x)}(\psi^{\bullet}w_i)=0$ for some $i\in\{1,2\}$. Applying Theorem \ref{ThmOver}, we find that for $x\in D(\F_p)\setminus\{0_J\}$, we have that
$$
m(x)\leq 2(\ord_{\psi^{-1}(x)}(\psi^{\bullet}w_i)+1)=2,
$$
since $\#\psi^{-1}(x)=1$.

Finally, let $x_1,x_2$ denote the points in $Z'\cap \nabla$, which are the preimages of $0_J$ via $\psi$. Since the supports of $w_1$ and $w_2$ are disjoint, there exists $w$ a linear combination of $w_1$ and $w_2$ such that $\ord_{x_j}(\psi^{\bullet}w)=0$. Applying Theorem \ref{ThmOver}, we find that $m(0_J)\leq 4$.

Since $m(x)<p-2$, Observation 
\ref{sharpness m} implies that $\#W_2(\Q_p)\cap\overline{J(\Q)}\cap U_x\leq m(x)+1$ for all $x\in W_2(\F_p)$. Therefore, we have
\begin{align}\label{finalcaseII}
\#W_2(\Q)&\leq \#W_2(\F_p)+2\#D(\F_p)+2\nonumber \\
&\leq \#W_2(\F_p)+2\#C(\F_p)+4.
\end{align}
The additional $2$ in \eqref{finalcaseII} is because $x_i$ could not belong to $Z'(\F_p)$.
\end{proof}
In what we do next, we work with an explicit choice of basis $\{w_1, w_2\}$ satisfying the condition of Proposition \ref{nonzero wedge}:
\begin{definition}Define $w_1$ and $w_2$ as follows:
\begin{align*}
w_1 &=(\beta_{12}x+\beta_{02})\frac{dx}{y}, \\
w_2 &= \begin{cases} (x^2-\beta_{01}/\beta_{12})\frac{dx}{y} &\text{if}\; \beta_{12} \neq 0, \\
(x^2+\beta_{01}x/\beta_{02})\frac{dx}{y} &\text{if}\; \beta_{12} = 0.\end{cases}
\end{align*}
\end{definition}

Case III is the generic situation and requires a separate, more technical analysis of singular and non-singular points. So for the remainder of this section, we assume that we are in Case III, i.e., that $Z=\mathbb{V}(\beta_{01}+\beta_{02}(x_1+x_2)+\beta_{12}x_1x_2)\subset C^{2}$ is irreducible. Let $\nu:\widetilde{Z}\to Z$ be the normalization map. In the following lemma, we compute $\ord_{z}(\nu^{\bullet}w)$, where $w$ is a differential on $C^{2}$ as in Proposition \ref{nonzero wedge}.  After that, we apply Riemann--Hurwitz to obtain the upper bound for $m(x)$. To state the lemma, we introduce some notation. Let us define the polynomial
\begin{equation}
F(x)=f\left(-\frac{\beta_{02}x+\beta_{01}}{\beta_{12}x+\beta_{02}}\right)(\beta_{12}x+\beta_{02})^{8}-(\beta_{02}^2-\beta_{01}\beta_{12})^4f(x).
\end{equation}
For $P=((a_1,b_1),(a_2,b_2))\in (C\setminus\{\infty\})^2$ a non-singular point of $Z$, we define $\delta_P$ as follows:
    \begin{equation}
    \delta_P=\begin{cases}
    1 & \text{if } (\beta_{12}a_2+\beta_{02})^4b_1-(\beta_{02}^2-\beta_{01}\beta_{12})^2b_2=0,\\
    0 & \text{otherwise.}
    \end{cases}
    \end{equation}
For $P\in Z$ a singular point of $Z$, let $t$ be an element in the local ring $\O_{P}$ such that $y_2t=y_1$. For each $z\in \nu^{-1}(P),$ we denote by $\mfrak_{z}$ the maximal ideal associated to $z$. We define $\gamma_P$ as follows:
    \begin{equation}
    \gamma_P=\begin{cases}
    1 & \text{if } t-t(P)\in \mfrak_{z},\\
    0 & \text{otherwise.}
    \end{cases}
    \end{equation}

\begin{lemma}\label{orders differential}
Suppose that $Z=\mathbb{V}(\beta_{01}+\beta_{02}(x_1+x_2)+\beta_{12}x_1x_2)\subset C^{2}$ is irreducible. Let $\nu:\widetilde{Z}\to Z$ be the normalization map. For every $P\in Z$, for every $z\in \nu^{-1}(P)$, and a suitable differential $w\in \Span_{\F_p}\{w_1,w_2\}$, we compute $\ord_{z}(\nu^{\bullet}w)$ as follows:
\small{
\begin{table}[H]
\centering
\begin{tabular}{|ccccc|}
\hline
\multicolumn{5}{|c|}{\textbf{$Z$ is non-singular at $P$}}                                        \\ \hline
\multicolumn{1}{|c|}{{Case}} &\multicolumn{1}{|c|}{{Hypothesis}} & \multicolumn{1}{c|}{$P\in Z\subset  C^2$} & \multicolumn{1}{c|}{$w$} & $\ord_{P}(\nu^{\bullet}w)$ \\ \hline
\multicolumn{1}{|c|}{\textbf{1.1}} &\multicolumn{1}{|c|}{none} & \multicolumn{1}{c|}{$((a_1,b_1),(a_2,b_2))\in (C\setminus\{\infty\})^2$} &  \multicolumn{1}{c|}{$w_1$} & $\delta_P \ord_{a_2}F(x)$ \\ \hline
\multicolumn{1}{|c|}{\textbf{1.2}} &\multicolumn{1}{|c|}{$\beta_{12}= 0$} & \multicolumn{1}{c|}{$((-\frac{\beta_{01}}{2\beta_{02}},b_1),(-\frac{\beta_{01}}{2\beta_{02}},b_2))$} & \multicolumn{1}{c|}{$w_2+aw_1$} & $\ord_{-\frac{\beta_{01}}{2\beta_{02}}}(x^2+\frac{\beta_{01}}{\beta_{02}}x+a\beta_{02})$ \\ \hline
\multicolumn{1}{|c|}{\textbf{1.3}} &\multicolumn{1}{|c|}{$\beta_{12}\neq 0$} & \multicolumn{1}{c|}{$((-\frac{\beta_{02}}{\beta_{12}},b_1),\infty)$, $(\infty,(-\frac{\beta_{02}}{\beta_{12}},b_2))$} & \multicolumn{1}{c|}{$w_2$} & $0$ \\ \hline
\multicolumn{5}{|c|}{\textbf{$Z$ is singular at $P$}}                                            \\ \hline
\multicolumn{1}{|c|}{{Case}} &\multicolumn{1}{|c|}{{Hypothesis}} & \multicolumn{1}{c|}{$P\in Z\subset  C^2$} &  \multicolumn{1}{c|}{$w$} & $\ord_{z}(\nu^{\bullet}w)$ \textbf{for} $z\in \nu^{-1}(P)$ \\ \hline
\multicolumn{1}{|c|}{\textbf{2.1}} &\multicolumn{1}{|c|}{$\beta_{12}= 0$} & \multicolumn{1}{c|}{$(\infty,\infty)$} & \multicolumn{1}{c|}{$w_2+aw_1$} & $0$ \\ \hline
\multicolumn{1}{|c|}{\textbf{2.2}} &\multicolumn{1}{|c|}{$\beta_{12}\neq 0$} & \multicolumn{1}{c|}{$((-\frac{\beta_{02}}{\beta_{12}},0),\infty)$, $(\infty,(-\frac{\beta_{02}}{\beta_{12}},0))$} & \multicolumn{1}{c|}{$\beta_{12}w_2$} & $2\gamma_{z} \ord_{-\frac{\beta_{02}}{\beta_{12}}}(F(x)-1)$ \\ \hline
\multicolumn{1}{|c|}{\textbf{2.3}} &\multicolumn{1}{|c|}{none} & \multicolumn{1}{c|}{$((\xi_1,0),(\xi_2,0))$ for $\xi_1\neq \xi_2$} & \multicolumn{1}{c|}{$w_1$} & $2\gamma_{z}\ord_{\xi_1}(F(x))-1$ \\ \hline
\multicolumn{1}{|c|}{\textbf{2.4}} &\multicolumn{1}{|c|}{$\beta_{12}= 0$} & \multicolumn{1}{c|}{$((\xi,0),(\xi,0))$} & \multicolumn{1}{c|}{$w_2+aw_1$} & $2\ord_{-\frac{\beta_{01}}{2\beta_{02}}}(x^2+\frac{\beta_{01}}{\beta_{02}}x+a\beta_{02})$ \\ \hline
\multicolumn{1}{|c|}{\textbf{2.5}} &\multicolumn{1}{|c|}{$\beta_{12}\neq 0$} & \multicolumn{1}{c|}{$((\xi,0),(\xi,0))$} & \multicolumn{1}{c|}{$w_1$} & $0$ \\ \hline
\end{tabular}
\end{table}}
\end{lemma}

\begin{proof}
To begin with, suppose that $P=((a_1,b_1),(a_2,b_2))$ is a non-singular point of $Z$; in particular, $b_1$ or $b_2$ differs from $0$. Without loss of generality, let us assume that $b_1\neq 0$, thus, $y_2-b_2$ is a local parameter at $z$ on $Z$ and $\nu^{\bullet}w_1$ is
\begin{equation*}
2\left(-\frac{(\beta_{02}^2-\beta_{01}\beta_{12})^2y_2}{(\beta_{12}x_2+\beta_{02})^3y_1}+(\beta_{12}x_2+\beta_{02})\right)\frac{dy_2}{f'(x_2)}.
\end{equation*}
Notice that if $b_2=0$, we have that $\ord_{((a_1,b_1),(a_2,b_2))}(\nu^{\bullet}w_1)=0$, since $a_2\neq -\beta_{02}/\beta_{12}$. If $b_2\neq 0$, $x_2-a_2$ is also a local parameter, and we have
\begin{align*}
\ord_{P}(\nu^{\bullet}w_1)&=\ord_{x_2-a_2}((\beta_{12}x_2+\beta_{02})^4y_1-(\beta_{02}^2-\beta_{01}\beta_{12})^2y_2)\nonumber\\
&=\delta_P \ord_{a_2}\left(f\left(-\frac{\beta_{02}x+\beta_{01}}{\beta_{12}x+\beta_{02}}\right)(\beta_{12}x+\beta_{02})^{8}-(\beta_{02}^2-\beta_{01}\beta_{12})^4f(x)\right)\nonumber\\
&=\delta_P \ord_{a_2}F(x),
\end{align*}
which proves Case 1.1. 
Let us now assume that $\beta_{12}=0$ ($\beta_{02}\neq 0$) and $-\beta_{01}/2\beta_{02}$ is not a root of $f$. Notice that $Z$ is non-singular at the points 
\begin{equation*}
((-\beta_{01}/2\beta_{02},\pm\sqrt{f(-\beta_{01}/2\beta_{02})}),(-\beta_{01}/2\beta_{02},\mp\sqrt{f(-\beta_{01}/2\beta_{02})})).    
\end{equation*}
We work separately in this case since they belong to the exceptional divisor $\nabla$.
In this case, $x_1+\beta_{01}/2\beta_{02}$ is a local parameter, and we have
\begin{equation*}
\nu^{\bullet}\left(w_2+aw_1\right)=2\left(x_1^2+\frac{\beta_{01}}{\beta_{02}}x_1+a\beta_{02}\right)(y_1-y_2)\frac{dx_1}{y_2y_1}.
\end{equation*}
Consequently, $\ord_{z}\nu^{\bullet}(w_2+aw_1)=\ord_{-\beta_{01}/2\beta_{02}}(x^2+(\beta_{01}/\beta_{02})x+a\beta_{02})$ proving Case 1.2.

Finally, if $\beta_{12}\neq 0$ and $-\beta_{02}/\beta_{12}$ is not a root of $f$, 
\begin{equation*}
((-\beta_{02}/\beta_{12},\pm\sqrt{f(-\beta_{02}/\beta_{12})}),\infty)\text{ and }(\infty,(-\beta_{02}/\beta_{12},\pm\sqrt{f(-\beta_{02}/\beta_{12})}))    
\end{equation*}
are non-singular points of $Z$. After the change of variables $(x_2,y_2)\mapsto (1/x_2,y_2/x_2^4)$ as in \cite[Exercise II.2.14]{SilvermanArithmetic}, $y_2$ is a local parameter and $\nu^{\bullet}w_2$ is 
\begin{equation*}
\left(-\frac{(x_1^2-\beta_{01}/\beta_{12})(\beta_{02}x_1+\beta_{01})y_2}{(\beta_{02}x_2+\beta_{12})y_1} +1-\frac{\beta_{01}x_2^2}{\beta_{12}}\right)\frac{2dy_2}{f_{*}'(x_2)}.
\end{equation*}
We use that $dx_1=-dx_2(\beta_{02}x_1+\beta_{01})/(\beta_{02}x_2+\beta_{12})$. Setting $x_1=-\beta_{02}/\beta_{12}$, $x_2=0$, and $y_2=0$, we obtain that 
\begin{equation*}
\ord_{((-\beta_{02}/\beta_{12},\pm\sqrt{f(-\beta_{02}/\beta_{12})}),\infty)}(\nu^{\bullet}w_2)=0,    
\end{equation*}
which proves Case 1.3 and then the lemma for non-singular points of $Z$.

Now, we compute the orders of the singular points of $Z$. For the sake of clarity, we outline the procedure for computing $\ord_{z}\nu^{\bullet}(w_2+(a/\beta_{02})w_1)$ for $z\in \nu^{-1}(\infty,\infty)$, where $a\in\F_p$. (For a more extensive discussion on differentials on singular curves, we refer the reader to \cite[Chapter IV]{SerreAlgebraic}.)

The change of variables $(x,y)\mapsto (1/x,y/x^4)$ turns $\nu^{\bullet}(w_2+(a/\beta_{02})w_1)$ into
\begin{equation}\label{dif}
\left(1+\frac{\beta_{01}x_1}{\beta_{02}}+a\beta_{02}x_1^2\right)\frac{-2dy_1}{f_{*}'(x_1)}+\left(1+\frac{\beta_{01}x_2}{\beta_{02}}+a\beta_{02}x_2^2\right)\frac{-2dy_2}{f_{*}'(x_2)},
\end{equation}
where $f_{*}(x)=x^8f(1/x)$. Let $g(x)\in \F_p[x]$ be the polynomial satisfying $f_{*}(x)=xg(x)$. With this change of variables, the equation defining $Z$ becomes $\beta_{02}(x_1+x_2)+\beta_{01}x_1x_2$, in particular, $x_2=-\beta_{02}x_1/(\beta_{01}x_1+\beta_{02})$. Let $\widetilde{\O}_{(\infty, \infty)}$ denote the integral closure of $\O_{(\infty, \infty)}$ in the function field $K(Z)$. Note that $t\in\widetilde{\O}_{(\infty, \infty)}$, where $t$ is defined by 
\begin{equation*}
t^2=\frac{-(\beta_{01}x_1+\beta_{02})g(x_1)}{\beta_{02}g(x_2)}.
\end{equation*}
Since $\#\nu^{-1}(\infty,\infty)=2$, we have that $\widetilde{\O}_{(\infty, \infty)}$ has two maximal ideals $\mfrak_1$ and $\mfrak_2$, above $\mfrak_{(\infty,\infty)}$, containing  $\mfrak_{(\infty,\infty)}+(t+\sqrt{-1})$  and $\mfrak_{(\infty,\infty)}+(t-\sqrt{-1})$, respectively. Since $y_1t=y_2$ we have that $y_1$ is a local parameter in each localization by $\mfrak_i$. Note that we have
\begin{equation*}
dy_2=-\frac{(\beta_{01}x_2+\beta_{02})tf_{*}'(x_2)}{(\beta_{01}x_1+\beta_{02})f_{*}'(x_1)}dy_1.
\end{equation*}
Therefore, \eqref{dif} turns into
\begin{equation*}
\left(1+\frac{\beta_{01}x_1}{\beta_{02}}+a\beta_{02}x_1^2 -\left(1+\frac{\beta_{01}x_2}{\beta_{02}}+a\beta_{02}x_2^2\right)\frac{\beta_{01}x_2+\beta_{02}}{\beta_{01}x_1+\beta_{02}}t\right)\frac{-2dy_1}{f_{*}'(x_1)}.
\end{equation*}
Setting $x_1=x_2=0$ and $t=\pm\sqrt{-1},$ we get that $\ord_{z}\nu^{\bullet}(w_2+aw_1)=0$ for every $z\in\nu^{-1}(\infty,\infty)$ proving Case 2.1.

Now, let us assume that $\beta_{12}\neq 0$ and $-\beta_{02}/\beta_{12}$ is a root of $f$. In this case, $((-\beta_{02}/\beta_{12},0),\infty)$ and $(\infty,(-\beta_{02}/\beta_{12},0)$ are singular points of $Z$, and we can write $f(x)=(x+\beta_{02}/\beta_{12})h(x)$ for some polynomial $h(x)\in\F_p[x]$. The change of variables given in \cite[Exercise II.2.14]{SilvermanArithmetic} turns $\nu^{\bullet}(\beta_{12}w_2)$ into
\begin{equation*}
((\beta_{02}x_1+\beta_{01})^4-(\beta_{02}^2-\beta_{01}\beta_{12})^2t)\frac{-2dy_2((\beta_{12}x_1^2-\beta_{01})}{f'(x_1)(\beta_{02}x_1+\beta_{01})^4},
\end{equation*}
where $t$ satisfies
\begin{equation*}
t^2=\frac{f(1/x_2)x_2^7}{(\beta_{02}x_1+\beta_{01})h(x_1)}.
\end{equation*}
We define $\gamma_z=1$ if $t-(\beta_{02}^2-\beta_{01}\beta_{12})^2/\beta_{12}^4\in\mfrak_{z}$ and $\gamma_z=0$ otherwise. 

Since $\ord_{y_1}(x_1+\beta_{02}/\beta_{12})=2$, we have that $\ord_{z}(\nu^{\bullet}(\beta_{12}w_2))$ is equal to
\begin{equation*}
2\gamma_z\ord_{x-\frac{\beta_{02}}{\beta_{12}}}\left(f\left(-\frac{\beta_{02}x+\beta_{01}}{\beta_{12}x+\beta_{02}}\right)(\beta_{12}x+\beta_{02})^{7}(\beta_{02}^2-\beta_{01}\beta_{12})^4-(\beta_{02}x_1+\beta_{01})^8h(x)\right).
\end{equation*}
Consequently, we have
\begin{equation*}
\ord_{z}(\nu^{\bullet}(\beta_{12}w_2))=2\gamma_z \ord_{-\frac{\beta_{02}}{\beta_{12}}}(F'(x))=2\gamma \ord_{-\frac{\beta_{02}}{\beta_{12}}}(F(x)-1),
\end{equation*}
for $z\in \nu^{-1}((-\beta_{02}/\beta_{12},0),\infty)$. This yields Case 2.2 of the lemma.

To prove Cases 2.3, 2.4, and 2.5, we assume that $\xi_1$ and $\xi_2$ are roots of $f$ satisfying 
$$
\xi_1=-(\beta_{02}\xi_2-\beta_{01})/(\beta_{12}\xi_2+\beta_{02}).
$$

First, we assume that $\xi_1\neq\xi_2$. In this case, we have $f(x)=(x-\xi_1)(x-\xi_2)q(x)$ for some $q\in \F_p[x]$ and
\begin{equation*}
\nu^{\bullet}w_1=(-(\beta_{02}^2-\beta_{01}\beta_{12})^2t+(\beta_{12}x_2+\beta_{02})^4)\frac{2dy_2}{f'(x_2)(\beta_{12}x_1+\beta_{02})^3},
\end{equation*}
where $t$ satisfies
\begin{equation*}
t^2=\frac{(\beta_{12}\xi_1+\beta_{02})(\beta_{12}\xi_2+\beta_{02})(\beta_{12}x_2+\beta_{02})^2 q(x_2)}{(\beta_{02}^2-\beta_{01}\beta_{12})^2q(x_1)}.
\end{equation*}
Since $\ord_{y_2}(x_2-\xi_2)=2$, for $z\in \nu^{-1}((\xi_1,0),(\xi_2,0))$ we have that $\ord_{z}(\nu^{\bullet}w_1)$ is equal to
\begin{equation*}
2\gamma_z\ord_{\xi_2}\left(q\left(-\frac{\beta_{02}x+\beta_{01}}{\beta_{12}x+\beta_{02}}\right)(\beta_{12}x+\beta_{02})^{6}-(\beta_{02}^2-\beta_{01}\beta_{12})^2(\beta_{12}\xi_1+\beta_{02})(\beta_{12}\xi_2+\beta_{02})q(x)\right),
\end{equation*}
where $\gamma_z=1$ if $t+(\beta_{12}\xi_2+\beta_{02})^4/(\beta_{02}^2-\beta_{01}\beta_{12})^2\in\mfrak_{z}$ and $\gamma_z=0$ otherwise. Therefore, we have
\begin{equation*}
\ord_{z}(\nu^{\bullet}w_1)=2\gamma_z \ord_{\xi_2}(F(x)/(x-\xi_1)(x-\xi_2)=2\gamma_z \ord_{\xi_2}(F(x))-1.
\end{equation*}
This yields Case 2.3 of the lemma.

For the remaining two cases, we assume that $\xi_1=\xi_2$. 

\noindent When $\beta_{12}=0$ (with $\beta_{02}\neq 0$), $-\beta_{01}/2\beta_{02}$ is a root of $f$. Therefore, we have that $f(x)=(x+\beta_{01}/2\beta_{02})g(x)$ for some $g(x)\in\F_p[x]$. In this case, $\ord_{y_1}(x_1+\beta_{01}/2\beta_{02})=2$ and $t$ satisfies $t^2=-g(x_1)/g(x_2)$, implying that
\begin{align*}
\ord_{z}\nu^{\bullet}(w_2+aw_1)&= 2\ord_{-\beta_{01}/2\beta_{02}}(x_1^2+\frac{\beta_{01}}{\beta_{02}}x_1+a\beta_{02})(g(x_1)+g(-x_1-\beta_{01}))\nonumber\\
&= 2 \ord_{-\beta_{01}/2\beta_{02}}(x_1^2+\frac{\beta_{01}}{\beta_{02}}x_1+a\beta_{02}),
\end{align*}
for $z\in \nu^{-1}((-\beta_{01}/2\beta_{02},0),(-\beta_{01}/2\beta_{02},0))$, proving Case 2.4.

Finally, when $\beta_{12}\neq0$, we have that $\xi$ is a common root of $\beta_{12}x^2+2\beta_{02}x+\beta_{01}$ and $f$. In this case, we have $f(x)=(x-\xi)q(x)$ for some $q(x)\in \F_q[x]$ and $\nu^{\bullet}w_1$ is
\begin{equation*}
(-(\beta_{02}^2-\beta_{01}\beta_{12})^2t+(\beta_{12}x_2+\beta_{02})^4)\frac{2dy_2}{f'(x_2)(\beta_{12}x_1+\beta_{02})^3},    
\end{equation*}
where $t$ satisfies
\begin{equation}\label{diagonal 2}
t^2=-\frac{(\beta_{02}^2-\beta_{01}\beta_{12})q(x_2)}{(\beta_{12}\xi+\beta_{02})(\beta_{12}x_2+\beta_{02})q(x_1)}.    
\end{equation}
Note that $\beta_{02}^2-\beta_{01}\beta_{12}=(\beta_{12}\xi+\beta_{02})^2$. Let $\widetilde{O}_P$ denote the integral closure of $O_P$.  Replacing $x_1$ and $x_2$ by $\xi$ in \eqref{diagonal 2}, we get that $t=\pm\sqrt{-1}$, which implies that $(-(\beta_{02}^2-\beta_{01}\beta_{12})^2t+(\beta_{12}x_2+\beta_{02})^4)$ is a unit in the localization of $\widetilde{O}_P$ at $z$, for any $z\in\nu^{-1}((\xi,0),(\xi,0))$. 
As a consequence, we have that
$\ord_{z}(\nu^{\bullet}w_1)=0$, which yields Case 2.5 and thus the lemma. 
\end{proof}

Let $S$ be the set of roots $\xi$ of $f$ such that $-(\beta_{02}\xi+\beta_{01})/(\beta_{12}\xi+\beta_{02})$ is also a root of $f$. We consider the polynomial $G(x)$ defined by 
\begin{equation}\label{def G}
F(x)=\prod_{\xi\in S}(x-\xi)G(x).
\end{equation}
For $\xi$ a root of $\beta_{12}x^2+2\beta_{02}x+\beta_{01}$, we define
\begin{equation}
\gamma_\xi=\begin{cases}
1 & \text{if $f(\xi)\neq 0$ and $(\zeta,\sqrt{f(\zeta)})\in C(\Q)$},\\
0 & \text{otherwise.}
\end{cases}
\end{equation}
If $\beta_{12}=0$, we define 
\begin{equation}    N=\sum\ord_{a_1}G(x)+\frac{1}{2}\gamma_{-\frac{\beta_{02}}{2\beta_{01}}}(\ord_{\frac{-\beta_{02}}{2\beta_{01}}}G(x)-1),
\end{equation}
where the sum runs over the points $[(a_1,b_1),(a_2,b_2)]$ in $ D(\F_p)\setminus(\Delta\cup\{0_J\})$ such that the equality $(\beta_{02})^4b_1-(\beta_{02}^2-\beta_{01})^2b_2=0$ holds.

If $\beta_{12}\neq0$, let us define
\begin{equation}
\eta=\begin{cases}
1 & \text{if $-\beta_{02}/\beta_{12}$ is a root of $f$},\\
0 & \text{otherwise,}
\end{cases}
\end{equation}
and
    \begin{equation}
    N=\sum\ord_{a_1}G(x)+\eta( \ord_{-\frac{\beta_{02}}{\beta_{12}}}G(x)-1)+\frac{1}{2}\left(\sum_{\xi}\gamma_{\xi}(\ord_{\xi}G(x)-1)\right),
    \end{equation}
    where the first sum runs over $[(a_1,b_1),(a_2,b_2)]\in D(\F_p)\setminus(\Delta\cup\{0_J\})$ satisfying $(\beta_{12}a_2+\beta_{02})^4b_1-(\beta_{02}^2-\beta_{01}\beta_{12})^2b_2=0$ and the second sum runs over the roots of $\beta_{12}x^2+2\beta_{02}x+\beta_{01}$.

Now using this lemma, we may prove the following:
\begin{proposition}[Case III]\label{bound D irreducible}
Suppose that $Z=\mathbb{V}(\beta_{01}+\beta_{02}(x_1+x_2)+\beta_{12}x_1x_2)\subset C^{2}$ is irreducible. Then, for every $x\in W_2(\F_p)$ we have $m(x)\leq 6$. Consequently, we can apply Theorem \ref{ThmOver} with any prime of good reduction for $C$ such that $p\geq 11$, and we obtain the following upper bound for $W_2(\Q)$:
\begin{equation}
\#W_2(\Q)\leq \#W_2(\F_p)+\#D(\F_p)+\#\Sing(D)(\F_p)+N.
\end{equation}
\end{proposition}

\begin{proof} 
By Theorem \ref{ThmOver}, if $z\in W_2\setminus D$, we have that $m(z)=0$. First of all, let us assume that $z\in D$ is a non-singular point and $z\notin \{0_J,[(-\beta_{02}/\beta_{12},f(-\beta_{02}/\beta_{12})),\infty)]\}$.
Then, there are $a_i,b_i\in \F_p^{alg}$ such that $z=\psi([(a_1,b_1),(a_2,b_2)])$ with at least one $b_i\neq0$. 

Suppose that a root $\xi$ of $\beta_{12}x^2+2\beta_{02}x+\beta_{01}$ is not a root of $f$. Let $$
z=\nu'^{-1}\left([(\xi,\pm\sqrt{f(\xi)}),(\xi,\pm\sqrt{f(\xi)})]\right),
$$ 
then applying Riemann--Hurwitz and using Case 1.1 of Lemma \ref{orders differential}, we have
\begin{equation}
m(z)\leq \ord_{z}\nu'^{\bullet}(w_1)+1=\frac{1}{2}(\ord_{(\xi,\pm\sqrt{f(\xi)}),(\xi,\pm\sqrt{f(\xi)})}(\nu^{\bullet}(w_1))-1)+1\leq 4.  
\end{equation}
Let us now assume that $a_1\neq a_2$. Then by Case 1.1 of Lemma \ref{orders differential} we have
\begin{equation*}
m(z)\leq \ord_{[(a_1,b_1),(a_2,b_2)]}(\nu'^{\bullet}w_a)=\delta_z\ord_{a_1}(F(x))+1.
\end{equation*}
Since $\ord_{a_1}(F(x))=\ord_{a_2}(F(x))$ we have that $\ord_{a_1}F(x)\le 3$, in particular, $m(z)\leq 4$.

Now, we find an upper bound $m(0_J)$. 
First, we assume that $\beta_{12}=0$ and notice that
\begin{equation*}
\psi^{-1}(0_J)=\left\{[(\infty,\infty)],\left[(-\beta_{01}/2\beta_{02},\sqrt{f(-\beta_{01}/2\beta_{02}}),(-\beta_{01}/2\beta_{02},-\sqrt{f(-\beta_{01}/2\beta_{02}})\right]\right\},
\end{equation*}
which implies that
\begin{equation*}
m(0_J)\leq \ord_{z_1}\nu'^{\bullet}(w_1+aw_2)+\ord_{z_2}\nu'^{\bullet}(w_1+aw_2)+2,
\end{equation*}
where $z_1$ and $z_2$ are the preimages via $\nu'$ of $\psi^{-1}(0_J)$. Let us fix $a\in\F_p$ such that $-\beta_{01}/2\beta_{02}$ is not a root of $x^2+(\beta_{01}/\beta_{02})x+a\beta_{02}$. If $f(-\beta_{01}/2\beta_{02})\neq 0$, we have $\delta_P=0$ in Case 1.1. On the other hand, if $f(-\beta_{01}/2\beta_{02})= 0$, then Case 2.4 of Lemma \ref{orders differential} implies that
\begin{equation*}
\ord_{z}\nu'^ {\bullet}\left(w_2+\frac{a}{\beta_{02}}w_1\right)=\ord_{-\frac{\beta_{01}}{2\beta_{02}}}\left(x_1^2+\frac{\beta_{01}}{\beta_{02}}x_1+a\beta_{02}\right)=0,    
\end{equation*}
for $z= \nu'^{-1}([(-\beta_{01}/2\beta_{02},0),(-\beta_{01}/2\beta_{02},0)])$. In addition, Case 2.1 of Lemma \ref{orders differential} implies
that $\ord_{z}\nu'^{\bullet}(w_2+(a/\beta_{02})w_1)=0$,
for $z= \nu'^{-1}([\infty,\infty])$. Consequently, in either case, we have that
$m(0_J)\leq 2$.

Now, suppose that $\beta_{12}\neq 0$ and let $\xi_1$ and $\xi_2$ be the two roots of $\beta_{12}x^2+2\beta_{02}x+\beta_{01}$. Since $\psi^{*}(w_1\wedge w_2)$ is not of the form \eqref{divisor constant}, we have $\xi_1\neq\xi_2$. Then we have
\begin{equation*}
\psi^{-1}(0_J)=\left\{[(\xi_1,\sqrt{f(\xi_1)}),(\xi_1,-\sqrt{f(\xi_1)})],[(\xi_2,\sqrt{f(\xi_2)}),(\xi_2,-\sqrt{f(\xi_2)})]\right\},
\end{equation*}
which implies that
\begin{equation*}
m(0_J)\leq \ord_{z_1}\nu'^{\bullet}(w_1)+\ord_{z_2}\nu'^{\bullet}(w_1)+2,
\end{equation*}
where $z_1$ and $z_2$ are the preimages via $\nu'$ of $\psi^{-1}(0_J)$. If $f(\xi_i)\neq 0$, we have that $\delta_P=0$ in Case 1.1. On the other hand, if $f(\xi_i)= 0$, we proved in Case 2.5 that 
$\ord_{z_1}\nu^{\bullet}(w_1)=0$. In either case, we have $m(0_J)\leq 2$.

Finally, let $x\in D$ be a singular point with $x\neq 0_J$. We assume we are in Case 2.2 of Lemma \ref{orders differential}. Let us denote by $z_1$ and $z_2$ the preimages of $[(-\beta_{02}/\beta_{12},0),\infty)]$ via $\nu'$. Therefore we have
\begin{align*}
m([(-\beta_{02}/\beta_{12},0),\infty)])\leq& \ord_{z_1}(\nu'^{\bullet}(\beta_{12}w_2))+\ord_{z_2}(\nu'^{\bullet}(\beta_{12}w_2))+2\nonumber\\
=&\ord_{-\frac{\beta_{02}}{\beta_{12}}}(F(x)-1)+2.
\end{align*}
In this particular case, $\deg(F)=7$ and due to the fact that $\beta_{12}x^2+2\beta_{02}x+\beta_{01}$ divides $F$, we have that $\ord_{-\frac{\beta_{02}}{\beta_{12}}}(F(x)-1)\leq 4$, from which it follows that $m(z)\leq 6$.

The remaining scenario is Case 2.3: here there are two different roots $\xi_1$ and $\xi_2$ of $f$ such that $\xi_1=-(\beta_{02}\xi_2-\beta_{01})/(\beta_{12}\xi_2+\beta_{02})$. Let us denote by $z_1$ and $z_2$ the preimages of $[(\xi_1,0),(\xi_2,0)]$ via $\nu'$. Then by Lemma \ref{orders differential}, we have
\begin{align*}
m([(\xi_1,0),(\xi_2,0)])&\leq \ord_{y_1}(\nu'^{\bullet}w_1)+\ord_{y_2}(\nu'^{\bullet}w_1)+2\nonumber\\
&=\ord_{\xi_1}(F(x)-1)+2. 
\end{align*}
Since $\ord_{\xi_1}(F(x))=\ord_{\xi_2}(F(x))$ and $\beta_{12}x^2+2\beta_{02}x+\beta_{01}$ divides $F$, we have that $\ord_{\xi_1}(F(x)-1)\leq 2$, so $m(z)\leq 4$.

We have just proved that $m(x)\leq 6$ for every $x\in D(\F_p)$. Therefore we can apply Proposition 9.14 in \cite{CaroPasten2021} using any prime $p\geq 7$. Additionally, if $p\geq 11$ by Observation \ref{sharpness m} we have that $ \#W_2(\Q_p)\cap\overline{J(\Q)}\cap U_x\leq m(x)+1$. 

Consequently, we have that $\#W_2(\Q)\leq \#W_2(\F_p)+\sum_{z\in W_2(\F_p)}m(z)$. Since $\#\widetilde{\psi}^{-1}(z)\leq 2$ for every $z\in W_2$, we have that
\begin{equation}\label{quadratic points}
\#W_2(\Q)\leq \#W_2(\F_p)+\#D(\F_p)+\#\Sing(D)(\F_p)+\sum_{z\in W_2(\F_p)}\sum_{y\in \widetilde{\psi}^{-1}(z)}\ord_{y}(\widetilde{\psi}^{\bullet}(w_z)),
\end{equation}
where $w_z$ is a linear combination of $w_1$ and $w_2$ that minimizes the quantity 
$$
\sum_{y\in \widetilde{\psi}^{-1}(z)}\ord_{y}(\widetilde{\psi}^{\bullet}(w)).
$$
With the above computations, one can check that 
\begin{equation*}
\sum_{x\in W_2(\F_p)}\sum_{y\in \widetilde{\psi}^{-1}}\ord_{y}(\widetilde{\psi}^{\bullet}(w_x))\leq N,
\end{equation*}
which applied to \eqref{quadratic points} yields the desired result.
\end{proof}

To conclude this section, we prove Theorem \ref{ThmMain}:

\begin{proof}[Proof of Theorem \ref{ThmMain}] First of all, suppose that we are in Case I or Case II. By Propositions \ref{bound p case 1} and \ref{bound p case 2} we have that
\begin{equation*}
\#W_2(\Q)\leq \#W_2(\F_p)+2\# C(\F_p)+4.
\end{equation*}
The Hasse--Weil bound for an absolutely irreducible smooth projective curve $X$ defined over $\F_p$ with genus $g(X)$ states that 
\begin{equation}\label{Hasse-Weil}
\#X(\F_p)\leq p+1+ 2g(X)\sqrt{p}.
\end{equation}
Consequently, we have that  
\begin{equation*}
\#W_2(\Q)\leq \#W_2(\F_p)+2(p+1+ 2g(C)\sqrt{p})+4=\#W_2(\F_p)+2p+ 12\sqrt{p}+6.
\end{equation*}
Since the reduction of $W_2$ does not contain elliptic curves over $\F_p^{\rm alg}$, Lemmas \ref{Z reducible} and \ref{D reducible} imply that the remaining scenario is Case III. By Remark \ref{Z' normalization}, we have that $\widetilde{Z'}$ is the normalization of $D$. By Proposition \ref{bound D irreducible}, we have that
\begin{equation*}
\#W_2(\Q)\leq \#W_2(\F_p)+\#D(\F_p)+\#\Sing(D)(\F_p)+ N.
\end{equation*}
Applying Riemann--Hurwitz to $\widetilde{\varphi}:\widetilde{Z}\to C$ and $\widetilde{\pi}:\widetilde{Z}\to \widetilde{Z'}$, we obtain that 
\begin{equation*}
g(\widetilde{Z})= 13-\#\Sing(Z),   
\end{equation*}
and since  $\widetilde{\pi}$ ramifies whenever $Z$ is non-singular, we have that 
\begin{equation*}
g(\widetilde{Z'})\leq 7-\#\Sing(Z)/2 -\epsilon,    
\end{equation*}
where $\epsilon=1$ when $Z$ is non-singular and $\epsilon=0$ otherwise. In particular, $g(\widetilde{Z'})\leq 6$. Consequently, the Hasse--Weil bound for singular curves \cite{Aubry} implies that 
\begin{equation*}
\#D(\F_p)\leq p+12\sqrt{p}+1.
\end{equation*}
By the definition of $N$ and $D$, we have that $N\leq 8$, $\#\Sing(D)(\F_p)\leq 5$, and
\begin{align*}
\#W_2(\Q)&\leq \#W_2(\F_p)+p+12\sqrt{p}+1+\#\Sing(D)(\F_p)+N\nonumber\\
&\leq\#W_2(\F_p)+p+12\sqrt{p}+14.
\end{align*}
Since $p\geq 11$, we obtain the desired result.
\end{proof}

\begin{remark}
Although Theorem \ref{ThmMain} is stated for $p\geq 11$, we can apply this method whenever $p>m(x)$ for every $x\in W_2(\F_p)$. We have chosen this lower bound for 
$p$ to obtain a more uniform result. In the examples, we will use $p=5$ and $p=7$.
\end{remark}


\section{The Algorithm} 
Here we put together the results of the previous sections and present the algorithm for computing an upper bound on $\#W_2(\Q)$.
\smallskip

\begin{algorithm}[Algorithm for an upper bound on $\#W_2(\Q)$]\label{mainalg} \quad\\
\textbf{Input:} 
\begin{itemize}\item $C$: a genus 3 hyperelliptic curve over $\Q$ given by an odd degree model with Jacobian rank $1$.
\item $p$: a prime at least 5 of good reduction for $C$.
\end{itemize}
\textbf{Output:} An upper bound on $\#W_2(\Q)$ or an error returned that $D$ (a divisor computed in the algorithm) contains elliptic curves.

\begin{itemize}
\item[1.] Compute $\omega_1$ and $\omega_2$ a basis for $\Ann(p,J(\Q))$ from \cite{Balakrishnan}.
\item[2.] Reduce $\omega_1$ and $\omega_2$ modulo $p$ and denote the reductions as $w_1$ and $w_2$.
\item[3.] Compute $g(x_1,x_2)=\beta_{01}+\beta_{02}(x_1+x_2)+\beta_{12}x_1x_2$ as in Lemma \ref{lemma divisor}.
\item[4.] Check if there exist $a,b\in \F_{p}^{\rm alg}$ such that $(bx_1-a)(bx_2-a)=g(x_1,x_2)$.

\begin{itemize}
\item[a.] If $a,b$ exist, then Lemma \ref{Divisor D} gives the divisor $D$, and we compute the bound as in Proposition \ref{bound p case 1}.
\item[b.] If not, we continue with Step 5.
\end{itemize}
\item[5.] Check if $Z$ is reducible.
By Lemma \ref{Z reducible}, $Z$ is reducible if and only if every branched point of $\varphi:Z\to C$ is singular. 
\begin{itemize}
\item[a.] If $Z$ is reducible, we factorize it as in Lemma \ref{Z reducible}.
\begin{itemize}
\item[I.] If 
$\beta_{12}\neq 0$ and $f'(-\beta_{02}/\beta_{12})=a_7(\beta_{02}^2-\beta_{12}\beta_{01})^3$,
 we are in the case of Lemma \ref{D reducible}, where $D$ contains elliptic curves, and we cannot handle this case. Exit with an error.
\item[II.] Otherwise, we compute the bound following Proposition \ref{bound p case 2}.
\end{itemize}
\item[b.] If $Z$ is irreducible, we continue with Step 6. 
\end{itemize}
\item[6.] We compute the singular points in $Z$ and the branched points of $p:\widetilde{Z}\to C$. Then compute the upper bound on $\#W_2(\Q)$,  following Proposition \ref{bound D irreducible}.
\end{itemize}
\end{algorithm}


\section{Examples}
Here we present one sharp example in each of Cases I, II, and III, computed using Algorithm \ref{mainalg} and Siksek's Theorem \ref{residue 0J}. We note that Siksek's method does not appear to be directly applicable with the primes we have chosen, as there exist residue disks containing more than one point in $W_2(\Q)$. Furthermore, in each example, this method does not allow us to determine $W_2(\Q)$ using primes
less than or equal to $19$.

\begin{example}[Case I] Consider the hyperelliptic curve
\[
C\colon y^2=x^7-14x^6+49x^5+x^2-16x+64.
\]
Applying the algorithm from \cite{Balakrishnan}, we determine that the set of rational points of $C$ is
\[
C(\Q)=\{ \infty, (0 , 8), (0 , -8), (7 , 1), (7, -1)\}.
\]
We will show that $W_2(\Q)$ is the following set:
\begin{align*}
W_2(\Q)=
\{& 0_J, [(0 , 8)-(\infty)], [(0 , -8)-(\infty)], [(7 , 1)-(\infty)], [(7, -1)-(\infty)],\\
&[2(0 , 8)-2(\infty)], [2(0 , -8)-2(\infty)], [2(7 , 1)-2(\infty)], [2(7, -1)-2(\infty)],\\
&[(7 , 1)+(0,8)-2(\infty)], [(7, -1)+(0,8)-2(\infty)]
\}.    
\end{align*}

The set of $\F_7$-rational points of $C$ is
\[
    C(\F_7)=\{ \overline{\infty}, \overline{(0 , 1 )}, \overline{(0 , 6 )}, \overline{(1 , 1)}, \overline{(1, 6)}, \overline{(3 , 0)},
\overline{(5 , 0)} \}.
\]
We have 
$$
J(\Q)=\langle [(7,1)-(\infty)] \rangle.
$$
Thus, the reduction modulo $7$ of $W_2(\Q)$ must be contained in the following set:
\begin{align*}
S=\{\overline{0_J}, [\overline{(0,\pm1)}-\overline{(\infty)}], [2\overline{(0,\pm1)}-2\overline{(\infty)}]\}. 
\end{align*}
The annihilator of $J(\Q)$ under the integration pairing is
spanned by
\begin{align*}
\omega_1 &= (3\cdot7 + 4\cdot7^2 + 7^3 + O(7^4)) \frac{dx}{y} + (1+6\cdot7 +6\cdot 7^2 + 7^3 +O(7^4)) \frac{xdx}{y},\\
\omega_2 &= (2\cdot7^2 +4\cdot 7^3 + O(7^4)) \frac{dx}{y} + (1 + 6\cdot7 + 6\cdot7^2 + 7^3 + O(7^4))\frac{x^2dx}{y}.
\end{align*}
Let $w_i$ denote the reduction modulo $7$ of $\omega_i$. Then we have
\begin{align*}
\psi^{*}(w_1\wedge w_2)&=(x_1x_2+(x_1+x_2)+1)(x_2-x_1)\frac{dx_1\wedge dx_2}{y_1y_2}\\
&=(x_1x_2)(x_2-x_1)\frac{dx_1\wedge dx_2}{y_1y_2}.
\end{align*}
Consequently, we are in Case I. Following the notation of the proof of Proposition \ref{bound p case 1}, we let $P=(0,1)$. The support of $D(\F_7)$ is the union of the two curves $C_{\overline{(0,1)}}$ and $C_{\overline{(0,-1)}}$. By Theorem \ref{ThmOver} we have that $m(\overline{0_J})\leq 4$, $m([2\overline{(0,\pm1)}-2\overline{(\infty)}])\leq 2$ (using the differential $w_1=\frac{xdx}{y}$) and $m([\overline{(0,\pm1)}-\overline{(\infty)}])\leq 1$ (using the differential $w_2=\frac{x^2dx}{y}$).

Note that 
$$
\{[(7,\pm1)-(\infty)],[(0,\pm8)-(\infty)]\}\subset W_2(\Q_7)\cap\overline{J(\Q)}\cap U_{[\overline{(0,\pm1)}-\overline{(\infty)}]}. 
$$
Thus 
$$
2\leq\#W_2(\Q_7)\cap\overline{J(\Q)}\cap U_{[\overline{(0,\pm1)}-\overline{(\infty)}]}\leq 1+\frac{6}{5}m([\overline{(0,\pm1)}-\overline{(\infty)}])= 2.2.
$$
Similarly, we have that
$$
\{[(7,\pm1)+(0,\pm8)-2(\infty)],[2(7,\pm1)-2(\infty)],[2(0,\pm8)-2(\infty)]\}\subset W_2(\Q_7)\cap\overline{J(\Q)}\cap U_{[2\overline{(0,\pm1)}-2\overline{(\infty)}]}, 
$$
then 
$$
3\leq\#W_2(\Q_7)\cap\overline{J(\Q)}\cap U_{[\overline{(0,\pm1)}-\overline{(\infty)}]}\leq 1+\frac{6}{5}m([2\overline{(0,\pm1)}-2\overline{(\infty)}])= 3.4.
$$
Finally, observe that
$$
\{[(7,1)+(0,-8)-2(\infty)],[(7,-1)+(0,8)-2(\infty)],0_J\}\subset W_2(\Q_7)\cap\overline{J(\Q)}\cap U_{\overline{0_J}}. 
$$
Thus
$$
3\leq\#W_2(\Q_7)\cap\overline{J(\Q)}\cap U_{\overline{0_J}}\leq 1+\frac{6}{5}m(\overline{0_J})= 5.8.
$$
We claim that $\#W_2(\Q_7)\cap\overline{J(\Q)}\cap U_{\overline{0_J}}=3$. For the sake of contradiction, assume that there exist $P_1,P_2\in C(\Q_7)\setminus C(\Q)$ such that $P_1\neq \iota(P_2)$, the reduction of $[(P_1)+(P_2)-2(\infty)]\in W_2(\Q_7)$ is $\overline{0_J}$, and 
\[
\int_{\infty}^{P_1}\omega+\int_{\infty}^{P_2}\omega=0,
\]
for every nonzero $\omega\in \Ann(7,J(\Q))$. By Theorem \ref{residue 0J}, the reductions of $P_1$ and $P_2$ should be $\overline{(0,1)}$ and $\overline{(0,-1)}$, respectively. Notice that for every $\omega\in \Ann(7,J(\Q))$ we have 
\[
\int_{(0,8)}^{P_1}\omega=-\int_{(0,-8)}^{P_2}\omega=\int_{(0,8)}^{\iota(P_2)}\omega\neq 0.
\]
In this setup, we want to prove that for $[(P)-(\infty)]$ in
$$
U_{[\overline{(0,1)}-\overline{(\infty)}]}\setminus\{[(0,8)-(\infty)],[(7,1)-(\infty)]\},
$$ 
the intersection of the kernels of $\lambda_\alpha-\lambda_\alpha(P)$ and $\lambda_\beta-\lambda_\beta(P)$ is just $P$. In this residue disk, we obtain the two power series
\begin{align*}
h(t) &= (5\cdot7 + O(7^3))t + (2 + 4\cdot7 + 4\cdot7^2 + O(7^3))t^2 +  (1 + 4\cdot7 + 7^2 + O(7^3))t^3 + \cdots \\
g(t) &= (7^2 +  O(7^3))t + (4\cdot7^2 + O(7^3))t^2 + (6 + 5\cdot7 + 4\cdot7^2 + O(7^3))t^3 +  \cdots 
\end{align*}
associated to the differentials $\omega_1$ and $\omega_2$, respectively. We aim to show that for each $t_0\in 7\Z_7\setminus\{0,7\}$, the functions $h_{t_0}(t)=h(t)-h(t_0)$ and $g_{t_0}(t)=g(t)-g(t_0)$ have just one common zero, specifically $t_0$. Looking at the Newton polygon of $h_{t_0}$, we notice that it has at most two roots in $7\Z_7$. We prove this by dividing this into three cases:
\begin{itemize}
    \item[(a)] Suppose that $\ord_7(t_0)=m\ge 2$, which implies $\ord_7(h(t_0))\geq m+1$ and $\ord_7(g(t_0))\geq m+2$. We have that $t_0=a_07^m+O(7^{m+1})$ with $a_0\in \{1,\dots,6\}$, and we have that $h(t)\equiv h_{t_0}(t) \text{ (mod }7^{m+1})$ and $g(t)\equiv g_{t_0}(t) \text{ (mod }7^{m+2})$. Using the first congruence, we notice that the other root of $h_{t_0}$ must be of the form $t_1\coloneqq 7+b_07^m+O(7^{m+1})$ with $b_0\in \{1,\dots,6\}$ and $a_0+b_0=7$. Using the second congruence, we notice that such a $t_1$ cannot be a root of $g_{t_0}(t)$.
    \item[(b)] Suppose that $\ord_7(t_0)=1$ and $t_0=a_07+O(7^2)$ with $a_0\in\{1,2,3,4,6\}$. If $a_0=1$, we have that the other solution $t_1$ of $h_{t_0}(t)$ has to satisfy that $\ord_7(t_1)\ge 2$, then we could apply item (a) to $t_1$ to obtain that $g(t_1)\neq g(t_0)$. Suppose that $a_0\neq 1$. By looking at the coefficients of $h$ and $g$, we notice that there is an $s_0\in  7\Z_7\setminus\{0,7\}$ such that 
    \begin{equation}\label{order7}
    3=\ord_7(g(s_0)-g(t_0))\leq \ord_7(h(s_0)-h(t_0)). 
    \end{equation}
    Let us consider the power series
    \[
    H(t)=g_{t_0}(s_0)h_{t_0}(t)-h_{t_0}(s_0)g_{t_0}(t).
    \]
    By \eqref{order7} and looking at the Newton polygon of $H$, it has at most $2$ solutions in $7\Z_7$, which are precisely $t_0$ and $s_0$. Then the only common solution of $h_{t_0}$ and $g_{t_0}$ is $t_0$.
    \item[(c)] Suppose that $\ord_7(t_0)=1$ and $t_0=5\cdot7+O(7^2)$. Since $h(t)\equiv h_{t_0}(t) (\text{mod }7^{2})$ we have that the other solution of $h_{t_0}$ must be of the form $t_1\coloneqq 3\cdot7+O(7^2)$. Then we can apply item (b) to $t_1$ to obtain that $g(t_1)\neq g(t_0)$.
\end{itemize}
\end{example}

\begin{example}[Case II] Let $C$ be the hyperelliptic curve defined by the equation
\[
y^2=f(x) = (x^3 - 2x^2 - 3x - 5)(x^4 - 5x^3 + 2x^2 - x - 2).
\]
Applying the algorithm developed in \cite{Balakrishnan}, we know that 
$C(\Q)=\{ \infty\}$.
The $\F_5$-points of $C$ are
\[
C(\F_5)=\{ \overline{\infty}, \overline{(0 , 0)}, \overline{(1, 0)}, \overline{(2, 0)},\overline{(3, 0)}, \overline{(4, 0)}\}.
\]
We claim that $W_2(\Q)= \{0_J\}\cup \{P, \iota(P)\}$, where
\begin{align*}
P=\left[\left(\sqrt{-1},5\right)+\left(-\sqrt{-1},5\right)-2(\infty)\right].
\end{align*}
In this case, we have 
\[
J(\Q)=\langle P,Q \rangle,
\]
where $Q$ is the $2$-torsion point associated to the cubic polynomial $x^3 - 2x^2 - 3x - 5$. 
Note that the reduction of $P$ modulo $5$ is $[\overline{(2, 0)}+\overline{(3, 0)}-2(\overline{\infty})]$ and the reduction of $Q$ is $[\overline{(0, 0)}+\overline{(1, 0)}+\overline{(2, 0)}-3(\overline{\infty})]$. 
Consequently, the reduction modulo $5$ of $W_2(\Q)$ is the following set:
\begin{align*}
\red_5(J(\Q))\cap W_2(\F_5)= \{\overline{0_J},[\overline{(2, 0)}+\overline{(3, 0)}-2(\overline{\infty})]\},
\end{align*}
where $\red_5$ denotes the reduction map $\red_5\colon J(\Q_5)\to J(\F_5)$.

The annihilator of $J(\Q)$ under the integration pairing is
spanned by
\begin{align*}
\omega_1 &= (2 + 4\cdot5 + 3\cdot5^3 + O(5^4)) \frac{dx}{y} + (3 + 2\cdot5 + 2\cdot5^2 + 5^3 + O(5^4)) \frac{xdx}{y},\\
\omega_2 &= (3 + 3\cdot5 + 3\cdot5^2 + 5^3 + O(5^4)) \frac{dx}{y} + (3 + 3\cdot5 + 5^2 + 2\cdot5^3 + O(5^4))\frac{x^2dx}{y}.
\end{align*}
Let $w_i$ denote the reduction modulo $5$ of $\omega_i$. Then we have
\begin{align*}
\psi^{*}(w_1\wedge w_2)= &(-x_1x_2+(x_1+x_2)+1)(x_2-x_1)\frac{dx_1\wedge dx_2}{y_1y_2}.
\end{align*}
In this case, $\beta_{12}=-1$ and $\beta_{02}=\beta_{01}=1$, and we also have that 
\[
f\left(-\frac{x+1}{1-x}\right)(1-x)^8\equiv 4x^7 + 2x^6 + x^5 + x^3 + 3x^2 + 4x\equiv -f(x) \text{ (mod }5).
\]
Additionally, note that
\[
f'(-\beta_{02}/\beta_{12})\equiv f'(1)\equiv2\not\equiv 2^3\equiv a_7(\beta_{02}^2-\beta_{12}\beta_{01})^3 \text{ (mod }5).
\]
Consequently, we are in Case II. Since $\gamma$ in \eqref{eq polynomials} is $-1$ we have that
\begin{equation*}
Z=\mathbb{V}_Z(y_1-2y_2(\beta_{12}x_1+\beta_{02})^{4})\cup  \mathbb{V}_Z(y_1-3y_2(\beta_{12}x_1+\beta_{02})^{4}), 
\end{equation*}
where both components have the same image via the morphism $\pi$ of the diagram \eqref{diagram W2}. Let $Z''$ be the curve in $W_2$ via the morphism $\psi$ and by Proposition \ref{bound p case 2} we have that $D=2\cdot Z''$. The proof of Proposition \ref{bound p case 2} shows that $m(z)\leq 2$ if $z\in Z''$ and $m(z)= 0$ otherwise.

Applying Theorem \ref{residue 0J}, one can show that the only element in $W_2(\Q)$ reducing to $\overline{0_J}$ is precisely $0_J$.

On the other hand, note that $z=[\overline{(2, 0)}+\overline{(3, 0)}-2(\overline{\infty})]$ is in $ Z''$.
Then by Observation \ref{sharpness m}, we have that 
\[
\#(W_2(\Q_5)\cap \overline{J(\Q)}\cap U_z)\leq \left\lfloor1+\frac{4}{3}\cdot 2\right\rfloor=3.
\]
Indeed, this is an equality since 
\[
W_2(\Q_5)\cap \overline{J(\Q)}\cap U_z\supset\{P, \iota(P),[(Q_1)+(Q_2)-2(\infty)]\},
\]
where $Q_1$ and $Q_2$ are the two torsion points in $C(\Q_5)$ reducing to $\overline{(2, 0)}$ and $\overline{(3, 0)}$, respectively. Consequently, the only points in $W_2(\Q)$ that reduce to $z$ are $P$ and $\iota(P)$.
\end{example}

\begin{example}
[Case III] Let $C$ be the hyperelliptic curve defined by the equation
\[
y^2=x(x^2+4) (x^2-4x-3) (x^2+4x+2).
\]
Applying the algorithm developed in \cite{Balakrishnan}, we know that
$C(\Q)=\{ \infty, (0 , 0)\}$.
We have that $C(\F_5)$ is the following set:
\[
C(\F_5)=\{ \overline{\infty}, \overline{(0 , 0)}, \overline{(1 , 0)}, \overline{(4 , 0)}\}.
\]
We claim that $W_2(\Q)= \{0_J\}\cup \{P_i: i=1,\dots,6\}$, where
\begin{align*}
P_1 &=\left[(0,0)-(\infty)\right],\\
P_2 &= \left[(Q)+(Q^\sigma)-2(\infty)\right],\\
P_3&=\left[\left(2\sqrt{-1},0\right)+\left(-2\sqrt{-1},0\right)-2(\infty)\right],\\
P_4&=\left[\left(-2+2\sqrt{2},0\right)+\left(-2-2\sqrt{2},0\right)-2(\infty)\right],\\
P_5&=\left[\left(2+\sqrt{7},0\right)+\left(2-\sqrt{7},0\right)-2(\infty)\right],\\
P_6 &=\iota P_2,
\end{align*}
with $Q=\left(\frac{-13+2\sqrt{-14}}{9},\frac{12560-7045\sqrt{-14}}{2187}\right)$ and $Q^\sigma$ the Galois conjugate of $Q$. 

In this case, we have 
\[
J(\Q)=\langle P_2,P_3,P_4,P_5 \rangle,
\]
with $P_2$ the generator of the free part of $J(\Q)$.
Note that the reduction modulo $5$ of $W_2(\Q)$ must be contained in  the intersection of the reduction modulo $5$ of $J(\Q)$ with $W_2(\F_5)$, which is the following set:
\begin{align*}
S\coloneqq \{&\overline{0_J},[\overline{(0,0)}-\overline{(\infty)}],[\overline{(1,0)}-\overline{(\infty)}],[\overline{(4,0)}-\overline{(\infty)}], [\overline{(0,0)}+\overline{(1,0)}-2\overline{(\infty)}],\\
&[\overline{(0,0)}+\overline{(4,0)}-2\overline{(\infty)}],[\overline{(1,0)}+\overline{(4,0)}-2\overline{(\infty)}], [\overline{(z,0)}+\overline{(4z+1,0)}-2\overline{(\infty)}],\\
&[\overline{(z+4,0)}+\overline{(4z,0)}-2\overline{(\infty)}]\}, 
\end{align*}
where $z$ satisfies $z^2 + 4z + 2 = 0$, cutting out the degree 2 extension $\F_{25}/\F_{5}$.

The annihilator of $J(\Q)$ under the integration pairing is
spanned by
\begin{align*}
\omega_1 &= (4 + 5 + 4\cdot5^3 + O(5^4)) \frac{dx}{y} + (3 + 2\cdot5 + O(5^4)) \frac{xdx}{y},\\
\omega_2 &= (4 + 4\cdot5 + 2\cdot5^2 + 2\cdot5^3 +O(5^4)) \frac{dx}{y} + (3 + 2\cdot5 + O(5^4))\frac{x^2dx}{y}.
\end{align*}
Let $w_i$ denote the reduction modulo $5$ of $\omega_i$. Then we have
$$
\psi^{*}(w_1\wedge w_2) = (3x_1x_2+2(x_1+x_2)+4)(x_2-x_1)\frac{dx_1\wedge dx_2}{y_1y_2}.$$

The divisor $D$ is in Case III, i.e., the divisor $D\in \Div(W_2)$, associated with the wedge product of $w_1$ and $w_2$, is irreducible. In this case, the support of $D(\F_5)\subset W_2(\F_5)$ is 
\[
\{\overline{0_J},[ \overline{(0,0)}+\overline{(1,0)}-2\overline{(\infty)}] ,[\overline{(2z+1, 2)},\overline{(3z+3,2)}-2\overline{(\infty)}],[\overline{(2z+1, 3)}+\overline{(3z+3,3)}-2\overline{(\infty)}]\}.
\]
By \eqref{prop914}, there is at most one point of $(W_2(\Q_5)\cap \overline{J(\Q)}$ reducing to $z$, whenever $z$ is not in the support of $D(\F_5)$. Note that $[(2\sqrt{-1},0)-(\infty)]$ and $[(-2\sqrt{-1},0)-(\infty)]$ belong to $\overline{J(\Q)}\cap C(\Q_5)$. 

Therefore, for each point $z$ in $S\setminus \{\overline{0_J},[ \overline{(0,0)}+\overline{(1,0)}-2\overline{(\infty)}]\}$, we have that 
$$
\#(W_2(\Q_5)\cap \overline{J(\Q)}\cap U_z)=1.
$$
Furthermore, applying Theorem \ref{residue 0J}, one can prove that the only element in $W_2(\Q)$ reducing to $\overline{0_J}$ is precisely $0_J$. In conclusion, the only points in $W(\Q)$ reducing to $S\setminus \{[ \overline{(0,0)}+\overline{(1,0)}-2\overline{(\infty)}]\}$ are $0_J$, $P_1$, $P_3$, $P_4$ and $P_5$.

Finally, for $z=[ \overline{(0,0)}+\overline{(1,0)}-2\overline{(\infty)}]\in W_2(\F_5)$, note that this point is a singular point of $D$. The polynomial $G$ defined in \eqref{def G} is
$$
G(x)= 2x^4 + 4x^3 + x^2 + 4x + 2.
$$
Since $0$ and $1$ are not roots of $G$, we have that $m(z)\leq 2$.  Then by Observation \ref{sharpness m}, we have that 
\[
\#(W_2(\Q_5)\cap \overline{J(\Q)}\cap U_z)\leq \left\lfloor1+\frac{4}{3}\cdot 2\right\rfloor=3.
\]
Indeed, this is an equality since 
\[
W_2(\Q_5)\cap \overline{J(\Q)}\cap U_z\supset\{P_2,P_6,[(2\sqrt{-1},0)+(0,0)-2(\infty)]\}.
\]
Consequently, the only points in $W_2(\Q)$ that reduce to $z$ are $P_2$ and $P_6$.
\end{example}


\bibliography{ChabautySurfacesBiblio}
\bibliographystyle{amsalpha}
\end{document}